\documentclass{article}

\usepackage{amsmath}
\usepackage{amsthm}
\usepackage{amssymb}
\usepackage{amscd}
\usepackage{ifthen}
\usepackage{a4wide}
\usepackage{epsfig,graphics,graphicx}
\usepackage{array,tabularx}
\usepackage{url}

\usepackage[text={6.5in,8in},centering]{geometry}

\usepackage[all,cmtip]{xy}
\usepackage[colorlinks=true,urlcolor=blue,linkcolor=blue,plainpages=false,pdfpagelabels
]{hyperref}

\newcommand{\details}[1]{$\,$\\***DETAILS {\bf #1}***\\}
\newcommand{\solutionsem}[1]{} 

\newcommand{\veranstaltung}{Techniques of Combinatorial Optimization}
\newcommand{\veranstalter}{Lauren\c tiu Leu\c stean}
\newcommand{\semester}{Winter Term 2014/2015}

\newcounter{uebno}

\newsavebox{\uebung}
\savebox{\uebung}{\raisebox{-11pt}{\parbox[t]{6cm}{\bf FMI, CS, Master I\\ \noindent \semester\\ \veranstaltung\\
                                          \veranstalter}}}

\newcounter{ct}

\newtheorem{theorem}{Theorem}[section]

\newtheorem{definition}[theorem]{Definition}
\newtheorem{proposition}[theorem]{Proposition}
\newtheorem{corollary}[theorem]{Corollary}
\newtheorem{lemma}[theorem]{Lemma}
\newtheorem{fact}[theorem]{Remark}
\newtheorem{exemplu}[theorem]{Example}
\newtheorem{exercise}{Exercise}
\newtheorem{notation}[theorem]{Notation}

\newcommand{\bdfn}{\begin{definition}}
\newcommand{\edfn}{\end{definition}}
\newcommand{\bthm}{\begin{theorem}}
\newcommand{\ethm}{\end{theorem}}
\newcommand{\bprop}{\begin{proposition}}
\newcommand{\eprop}{\end{proposition}}
\newcommand{\bcor}{\begin{corollary}}
\newcommand{\ecor}{\end{corollary}}
\newcommand{\blem}{\begin{lemma}}
\newcommand{\elem}{\end{lemma}}
\newcommand{\bfact}{\begin{fact}}
\newcommand{\efact}{\end{fact}}
\newcommand{\bex}{\begin{exemplu}\begin{rm}}
\newcommand{\eex}{\end{rm}\end{exemplu}}
\newcommand{\bxc}{\begin{exercise}}
\newcommand{\exc}{\end{exercise}}
\newcommand{\bntn}{\begin{notation}}
\newcommand{\entn}{\end{notation}}

\newcommand{\be}{\begin{enumerate}}
\newcommand{\ee}{\end{enumerate}}
\newcommand{\bce}{\begin{center}}
\newcommand{\ece}{\end{center}}
\newcommand{\bi}{\begin{itemize}}
\newcommand{\ei}{\end{itemize}}
\newcommand{\bt}{\begin{tabular}}
\newcommand{\et}{\end{tabular}}
\newcommand{\beq}{\begin{equation}}
\newcommand{\eeq}{\end{equation}}
\newcommand{\ba}{\begin{array}} 
\newcommand{\ea}{\end{array}}
\newcommand {\bea} {\begin{eqnarray}}
\newcommand {\eea} {\end {eqnarray}}
\newcommand {\bua} {\begin{eqnarray*}}
\newcommand {\eua} {\end {eqnarray*}}
\newcommand{\bcev}{\vspace{-0.2cm}\begin{center}}
\newcommand{\ecev}{\end{center}\vspace{-0.2cm}}
\newcommand{\bceuv}{\vspace{-0.5mm}\begin{center}}
\newcommand{\eceuv}{\end{center}\vspace{-0.5mm}}

\newcommand{\Ra}{\Rightarrow}

\newcommand{\Lra}{\Leftrightarrow}

\newcommand{\se}{\subseteq}

\newcommand{\ds}{\displaystyle}

\newcommand{\ol}{\overline}

\newcommand{\lan}{\left\langle}
\newcommand{\ran}{\right\rangle}

\def\N{{\mathbb N}}

\def\R{{\mathbb R}}

\newcommand{{\Fdo}}{F}

\usepackage{soul}

\newcommand{\WKL}{{\sf WKL}_0}
\newcommand{\RCA}{{\sf RCA}_0}

\newcommand{\Br}{\mathcal{T_B}}
\newcommand{\Ms}{\mathcal{T_M}}
\newcommand{\bAC}{\sf{bAC}}
\newcommand{\BC}{\sf{bC}}
\newcommand{\bfi}{{\rm bfi}}
\newcommand{\foralltilde}{\tilde{\forall}}

\newtheorem*{metatheorem}{First metatheorem}
\newtheorem*{metatheorem_2}{Second metatheorem}
\newtheorem*{metatheorem_3}{Third metatheorem}

\newcommand{\usf}{\omega}
\newcommand{\usw}{\gamma}
\newcommand{\usr}{r}

\newcommand{\fseco}[1]{\widehat{#1}}
\newcommand{\fsec}[1]{\ol{#1}}

\renewcommand{\details}[1]{}

\title{On the removal of weak compactness arguments in proof mining}
\author{Fernando Ferreira${}^{a}$, Lauren\c{t}iu Leu\c{s}tean${}^{b,c,d}$  and Pedro Pinto${}^{a}$\\[2mm]
\footnotesize ${}^{a}$ Departamento de Matem\'atica, Faculdade de Ci\^encias, Universidade de Lisboa,\\ 
\footnotesize Campo Grande, Ed. C6, 1749-016 Lisboa, Portugal \\ [1mm]
\footnotesize ${}^{b}$ The Research Institute of the University of Bucharest (ICUB), University of Bucharest,\\
\footnotesize Bd. M. Kog\u{a}lniceanu 36-46, 050107, Bucharest, Romania\\[1mm]
\footnotesize ${}^{c}$Faculty of Mathematics and Computer Science, University of Bucharest,\\
\footnotesize  Academiei 14, 010014, Bucharest, Romania\\[1mm]
\footnotesize ${}^d$Simion Stoilow Institute of Mathematics of the Romanian Academy,\\
\footnotesize Calea Grivi\c tei 21, 010702 Bucharest, Romania \\[2mm]
\footnotesize E-mails:  fjferreira@fc.ul.pt, laurentiu.leustean@unibuc.ro, pedrosantospinto@hotmail.com
}
\date{}

\begin{document}

\maketitle

\begin{abstract} The main observation of this paper is that some sequential weak compactness arguments in Hilbert space theory can be replaced by Heine/Borel compactness arguments  (for the strong topology). Even though the latter form of compactness fails in (infinite-dimensional) Hilbert spaces, it nevertheless trivializes under the so-called bounded functional interpretation. As a consequence, the proof mining program of extracting computational bounds from ordinary proofs of mathematics can be applied to {\em modified proofs} which use these false Heine/Borel compactness arguments. Additionally, the bounded functional interpretation provides good logical guidance in formulating quantitative versions of analytical statements. We illustrate these claims with three minings. The bounded  functional interpretation is here used for the first time in proof mining. \\



\end{abstract}

\section{Introduction}

Proof mining is a research program whose aim is to extract computational information from proofs in mathematics. There are two main aspects in the practice of proof mining. One relevant aspect is finding a quantitative version of the theorem under analysis. For instance, many theorems analyzed with the proof mining methodology say that, under certain circunstances, a given sequence converges. This is the case with the three results mined in this paper. Suppose that the theorem asserts that a given sequence $(u_n)_{n\in \N}$ of elements of a Hilbert space converges. This is equivalent to saying that the sequence enjoys the Cauchy property:

$$\forall k \in \N \, \exists n \in \N \, \forall m \in \N \,\forall i,j \in [n,m] \left( \Vert u_i - u_j \Vert < \frac{1}{k+1} \right),$$

where $[n,m] := \{n,n+1,\ldots,m\}$, for any $n,m\in\N$. In general, there is no computable rate of Cauchyness, i.e., there is no computable numerical function $h: \N \to \N$ such that

$$\forall k \in \N \, \forall m \in \N \, \forall i,j \in [h(k),m] \; \left( \Vert u_i - u_j \Vert < \frac{1}{k+1} \right)$$

There are, in fact, already counterexamples for this assertion in $\R$ with computable Cauchy sequences of rational numbers. Instead of the Cauchy property, we consider a quantitative version thereof:

$$\forall k \in \N , f \in \N^\N \, \exists n \in \N \, \forall i,j \in [n,f(n)] \; \left( \Vert u_i - u_j \Vert < \frac{1}{k+1} \right)$$

This is called the {\em metastable} version of the Cauchy property. The reader can stare a bit at the Cauchy property and its metastable version and convince himself that they are equivalent (suggestion: fix $k$ and negate what follows). A proof mining analysis provides a concrete computable functional $\phi: \N \times \N^\N \to \N$ (a rate of metastability) such that

$$\forall k \in \N, f \in \N^\N \, \exists n \leq \phi(k,f) \,\, \forall i,j \in [n,f(n)] \; \left( \Vert u_i - u_j \Vert < \frac{1}{k+1} \right)$$

In the proof mining analyses of this paper, we obtain concrete computable functionals as above for the sequences under consideration. This is the other main aspect of a proof mining analysis. Why is this possible? For instance, if the proof of the theorem is formalizable in suitable systems based on finite-type Peano arithmetic, then a (computable) rate of metastability is ensured {\em a priori} by a logical metatheorem and (in this situation) the associated computable functionals are rather simple conceptually (they come from the so-called primitive recursive functionals in the sense of G\"odel). In actual practice, a proof mining analysis can ultimately be done without knowing in detail its underlying theoretical basis and the functionals that appear are typically obtained using very familiar mathematical constructions, like iterations.

Compactness results are the bread and butter of the analyst and they are used in many equivalent ways in mathematics without a second thought. However, in proof mining studies, there are certain distinctions that must be made because they are of crucial importance for a computational analysis. Take, for instance, the compactness of the closed unit interval. From the proof mining point of view, it makes a substantial difference {\em how} the compactness argument is used. One can use the fact that a sequence of real numbers of the closed unit interval has a convergent subsequence. This is a sequential compactness argument. One can also use the fact that every open covering of the closed unit interval has a finite subcovering. This is the Heine/Borel covering principle. Even though these two formulations of compactness are equivalent, they are -- from the proof mining point of view -- quite dissimilar. Heine/Borel compactness can be dealt within the simple framework of the primitive recursive functionals, whereas sequential compactness needs more sophisticated functionals. More importantly, within certain formal systems, a deeper analysis shows that Heine/Borel compactness arguments can be removed from arguments whose conclusion is a quantitative statement. This is emphatically not the case with sequential compactness arguments.

In Hilbert spaces, there are two important topologies at play: the strong and the weak topologies. Both coincide in having the same closed convex subsets. However, the topologies do come apart in infinite dimensional Hilbert spaces (they are the same in finite dimension). What does it mean for a sequence $(u_n)_{n\in \N}$ of elements of a Hilbert space to converge (to an element $u$) in any one of these topologies? It converges strongly if $\lim_n \Vert u_n - u \Vert = 0$. It converges weakly if, for every element $v$ of the space, the sequence of real numbers $(\langle u_n,v\rangle)_{n\in \N}$ converges to the real number $\langle u, v\rangle$. (Here, $\langle \cdot, \cdot \rangle$ is the inner product of the Hilbert space.) It is not in general the case that a sequence of elements of a bounded closed convex subset of a Hilbert space has a subsequence which converges strongly to a point of the subset. For instance, the closed unit ball of an infinite dimensional Hilbert space is never compact for the strong topology. It is, however, a classical result of Hilbert space theory that bounded closed convex subsets of a Hilbert space are sequentially compact for the weak topology. This means that every sequence of elements of a bounded closed convex subset of a Hilbert space has a subsequence which converges weakly to a point of that subset.

In this paper, we analize a famous strong convergence theorem of Felix Browder. Browder's original proof in \cite{Browder(67)} uses a sequential weak compactness argument. As commented before, this poses a problem for a perspicuous proof mining analysis. The fundamental observation of this paper is twofold. Browder's sequential weak compactness argument can be replaced by a Heine/Borel compactness argument for the strong topology. More specifically, one can prove Browder's theorem in a certain formal theory using the Heine/Borel covering principle $$\forall x \in C \, \exists n \in \N\,  (x \in \Omega_n) \to \exists n \in \N \, \forall x \in C \, \exists k \leq n \, (x \in \Omega_k),$$ where  $C$ is a bounded closed convex subset of a Hilbert space, and $(\Omega_n)_{n\in \N}$ is a sequence of open sets for the strong topology. Even though the above Heine/Borel covering principle is false in (infinite dimensional) Hilbert spaces, its use within certain formal theories can be {\em removed} from proofs of quantitative statements. This is the other aspect of the observation. The fact that this removal is {\em a priori} possible is explained by the so-called bounded functional interpretation, introduced in \cite{FerreiraOliva(05)} by the first author of this paper and Paulo Oliva. More to the point, one mines a {\em modified proof} of Browder's theorem, knowing that the above Heine/Borel covering principle trivializes under the bounded functional interpretation.

The paper is organized as follows. In the next section, we look into Browder's proof of his strong convergence result. We briefly recall this proof and show how the sequential weak compactness argument of the proof can be replaced by a Heine/Borel argument. (If the reader feels like it, he can jump now to the beginning of Section \ref{modified} to see this modified proof.) Next, we introduce a formal theory in which the modified proof of Browder's theorem can be formalized. We tried to be brief and rely on the available literature but, of course, this part is unavoidably technical. We state a metatheorem for this theory which guarantees the {\em a priori} existence of computational bounds and justifies the trivial use of the above Heine/Borel covering principle.  Later on, we exemplify this technique with the concrete minings (obtainment of rates of metastability) of three theorems in nonlinear analysis:  the already mentioned strong convergence theorem of Browder \cite{Browder(67)}, the strong convergence of Wittmann \cite{Wittmann(92)} for nonexpansive iterations and Bauschke's generalization \cite{Bauschke(96)} of Wittmann's result  to families of nonexpansive mappings. 
Browder's and Wittmann's theorems have already been mined by Ulrich Kohlenbach in \cite{Kohlenbach(11)}, where he analyzes  two proofs of Browder's theorem:  Browder's original proof, and an alternative proof of Benjamin Halpern \cite{Halpern(67)} which does not rely on a weak compactness argument.
Generalizations of Browder's and Wittmann's theorems have also  been mined by the second author and Kohlenbach for CAT(0) spaces \cite{KohlenbachLeustean(12), KohlenbachLeustean(14)}, by the second author  and Adriana Nicolae for CAT$(\kappa)$ spaces (with $\kappa >0$) \cite{LeusteanNicolae(16)} and, recently,  by Kohlenbach and Andrei Sipo\c s for uniformly smooth and uniformly convex Banach spaces \cite{KohlenbachSipos(19)} (completing a partial analysis of Wittmann's theorem from \cite{KohlenbachLeustean(12A)}).  A  metastable version of a generalization of Bauschke's theorem has been obtained, using also proof mining methods and extending \cite{Kohlenbach(11)}, by Daniel K\"ornlein \cite{Koernlein(16)} (see also K\"ornlein's PhD thesis \cite{Koernlein(16a)}).\\
The mining of the paradigmatic projection argument is done in Section~\ref{projection} using the bounded functional interpretation, and this is a novelty. In Section~\ref{generalprinciple}, we prove a result that isolates the Heine/Borel technique for the minings of this paper. We also state a general principle and obtain a general quantitative result that can be used in several situations. In Section~\ref{browder-wittmann}, we finish the first two minings, obtaining in this way quantitative results which are numerically similar to those of Kohlenbach \cite{Kohlenbach(11)}. Section \ref{section_6} is dedicated to the mining of Bauschke's strong convergence result \cite{Bauschke(96)}. We also use the bounded functional interpretation for doing it. Firstly, the general principle is widened to cover the new situation.  Afterwards, some necessary estimates are worked out in detail for this new case.  These results are then finally put together in order to conclude the mining.

\subsection{Further (and more specialized) introductory remarks}

The theoretical basis underlying the current practice of proof mining  rests on a modification of G\"odel's  functional {\em dialectica} interpretation \cite{Goedel(58)} -- the so-called monotone functional interpretation of Kohlenbach, introduced in \cite{Kohlenbach(96a)} -- and focuses on the extraction  of computational bounds,  as opposed to precise witnesses. Working with the monotone functional interpretation makes possible some of the most  distinctive features of proof mining, viz. its uniformity results or the simple analysis of certain forms of compactness. A fundamental milestone of proof mining was the introduction of abstract types, first in \cite{Kohlenbach(05)} (for the metric bounded case)  and then in \cite{GerhardyKohlenbach(08)}  (for the unbounded case). With this extension, a new base type $X$ is added to the base type of  the natural numbers. The new base type stands for an abstract metric space, but one can choose to be  more specific and consider normed, hyperbolic, Hilbert, CAT(0) spaces, etc. Special conceived variants  are also very useful, as we will illustrate.

The introduction of new base (abstract) types has several advantages. For instance, we are no longer  restricted to state theorems for ``computable" or ``representable" spaces only (as it is usually done in constructivism).  In many cases, the extraction of computable bounds makes sense for arbitrary spaces. Extractions are obtained  for genuinely non-denumerable data (e.g, nonseparable Hilbert spaces), that could not -- even in principle --  be ``computationally representable." With the new base types, we do not have to work with tedious representations of data. We deal with the objects directly, as points of an abstract space (as does the ordinary mathematician). Another important advantage is that, by abstracting from the representations of the mathematical objects, the logical form of the analyzed statements simplifies. As a consequence, we can formulate better quantitative versions of these statements (with a view to working out concrete bounds). These two advantages do not operate in isolation. The abstraction from concrete representations and the attending simplification of logical form work in tandem and, in fact, are inextricably intertwined.

That notwithstanding, the main advantage of the introduction of new base abstract types lies elsewhere. It lies in the abstract axiomatization. The axioms have the crucial instrumental role of delimiting precisely what is used in a proof. This is important because the use of certain principles may obstruct a computational analysis. In order to be able to extract computational information, certain hypotheses cannot be used even though, in the end, the extraction also applies to structures verifying those hypotheses. The most dramatic example of this phenomenon is that full extensionality cannot be used in proof mining. This is related to the well-known fact that full extensionality does not have a {\em dialectica} interpretation in G\"odel's ${\sf T}$ (G\"odel's ${\sf T}$ is a version of finite-type arithmetic). The reader can find in pp. 394--395 of \cite{Kohlenbach(08)} a practical discussion of these matters. Of course, all ordinary mathematical structures satisfy full extensionality and the minings apply to them. A less dramatic example, but an important one for applications, is that the use of the separability hypothesis in Hilbert spaces (the statement that the space has a countable dense subset) in some computational analyses unduly restricts the range of applications to finite-dimensional spaces (see p. 443 of \cite{Kohlenbach(08)}). On the other hand, an analysis that does not use separability applies to all Hilbert spaces, separable and nonseparable alike.

The formulation of good quantitative versions of mathematical statements can be difficult to uncover by the ordinary mathematician because it depends on the {\em logical form} of the statements. This particular point has often been emphasized by Kohlenbach, most recently in \cite{Kohlenbach(ta2)}. Ordinary mathematicians are very good at understanding what a mathematical claim {\em claims}, e.g., on what assumptions a claim rests, but are for the most part blind to the logical forms of the statements themselves. Within the context of a different discussion, Georg Kreisel has commented on this state of affairs in \cite{Kreisel(67)} when he wrote that ``in his own work [the ordinary mathematician] never gives a second thought to the form of the predicate in a comprehension axiom! (This is the reason why, e.g., Bourbaki is extremely careful to isolate the assumptions of a mathematical theorem, but never (\ldots) what instances of the comprehension axioms are used. (\ldots))" Logical guidance plays an important role in arriving at good formulations of quantitative versions of analytical statements. In proof mining, the logical guidance is given by the functional interpretation. The monotone functional interpretation and the bounded functional interpretation differ in their formulations. The difference between the two formulations does not show up in the analysis of simple statements. For instance, they coincide in the quantitative formulation of Cauchyness (both give the metastability formulation discussed before). However, significant differences do appear in more involved statements. In this paper, these differences can be seen clearly in the formulation of the quantitative version of the projection argument (see the beginning of Section \ref{projection}).

A concrete mining of a mathematical result is just a piece of ordinary mathematics. No logic is needed in laying, or explaining, the arguments and computations of a piece of mining. However,  the arguments and computations are not {\em ad hoc}. They follow a method. The method is based  on certain metatheorems that have the following general form: If such-and-such theorem of  ordinary mathematics is provable in a such-and-such axiomatic (formal) system, then  such-and-such type of computational bounds can be extracted from the (formal) proof.  The actual mining exhibits concrete computational bounds. The metatheorems ensure {\em a priori} that if an ordinary proof of mathematics can be formalized in a certain axiomatic system, then a certain computational bound exists. In practice, the proof mining researcher does not  formalize the ordinary mathematician's proof. Rather, he convinces himself that the proof can  be formalized within a certain axiomatic system and proceeds swiftly with the extraction of the  bounds. The method of extraction follows the proof of the metatheorems in the following sense: one associates to each step of the proof its quantitative formulation (given by the functional interpretation) and, at the same time, one finds concrete computational bounds realizing the formulation. This is always possible because logical transitions preserve this association. Given that  the ordinary mathematician's proof is not formal, this is done rather loosely in practice but,  in the end, one does obtain a concrete bound and a rigorous mathematical proof that the bound does the job that it is supposed to do. (See \cite{Kohlenbach(ta1)} for further discussions on related proof-theoretic issues and, especially, for examples.)

As a rule, ordinary mathematicians do not care much about which principles they use in a proof. However, as we have already observed, what is used in a proof is of paramount importance for the enterprise of proof mining, not only as a question of principle (to ensure the very existence of uniform bounds)  but also for knowing what kind of mining analyses and bounds can be expected. The metatheorems of \cite{Kohlenbach(05)} and  \cite{GerhardyKohlenbach(08)} apply to very strong systems (using the so-called bar recursive functionals  of Clifford Spector \cite{Spector(62)}). However, bar recursive functionals seldom appear in a proof mining analysis -- only the much more simple primitive recursive functionals  of G\"odel's {\sf T} do appear. This is because certain forms of comprehension and some principles of choice are not used essentially in the proofs that have so far been mined (it is an interesting question to ask whether this phenomenon is due to a selective choosing in actual  research, or if it is mostly like that in ordinary mathematics). For instance, mathematicians  take infima of positive real sequences without a second thought. From the logical point of view, this is equivalent to a certain form of comprehension (technically, arithmetical comprehension). Even though this form of comprehension falls within the scope of the general metatheorems in \cite{Kohlenbach(05)} and \cite{GerhardyKohlenbach(08)}, its analysis uses bar recursive functionals. What often happens is that taking  infima is not essential to the proofs of computational relevant facts (one only needs to be as close to the infimum as one wants). The proof mining researcher analyzes instead a modified proof, one that avoids some of the spurious principles that are used by the ordinary mathematician.

A good example of this situation is the proof of Browder's strong convergence result of \cite{Browder(67)}. We will look carefully at the ordinary proof of this theorem, its modification and the corresponding mining. This will be done in the next section, in Section~\ref{projection} and in Section \ref{browder-wittmann}. The argument relies on the existence of certain fixed points whose existence could be problematic but, in fact, is not. Subsequently, it uses a projection argument where one must take an infimum and apply a certain strong (from the logical point of view)  form of choice. These procedures turn out not to be essential to the proof. However, the second part of Browder's proof raises a new problem because it uses a sequential weak compactness argument. There seemed to be no way of avoiding this argument in Browder's proof. In spite of the presence of  this argument, a successful mining was obtained in \cite{Kohlenbach(11)}. According to the section of acknowledgements of that paper, Eyvind Briseid pointed out that the use of sequential weak compactness in that mining has a trivial solution. There is no need to use any real strength of sequential weak compactness and, therefore, no need to rely on a complicated [{\em sic}] bar recursive solution. Is there a theoretical explanation for this situation? As discussed in the first part of this introduction, the bounded functional interpretation provides a theoretical explanation for this situation. One can see the mining of Browder's proof as applying to a modified proof which uses Heine/Borel compactness instead of a sequential weak compactness argument. A form of Heine/Borel covering principle trivializes under minings operated by the bounded functional interpretation and, additionally, it is removed from proofs of quantitative statements.

Before concluding this introduction, let us make a logical comment. The fact that  the bounded functional interpretation trivializes a form of the Heine/Borel covering principle in Hilbert  spaces has the consequence that this principle is conservative over a suitable base theory  with respect to a certain class of formulas (which includes the $\Pi^0_2$-formulas). The conservation result of the metatheorem of Subsection \ref{first_metatheorem} (or, even better, of the metatheorem of Sub-subsection \ref{second_metatheorem}) is, in fact, similar to the the well-known result of reverse mathematics that the theory $\WKL$ is $\Pi^0_2$-conservative over the base theory $\RCA$ (for reverse mathematics, see \cite{Simpson(99)}). The curious difference, as already observed, is that whereas weak K\"onig's lemma is true, the Heine/Borel covering principle is false (in infinite-dimensional Hilbert spaces). That notwithstanding, according to the bounded functional interpretation, the explanatory root of these two conservation results is the same.

\mbox{}

{\bf Notation:} $\N^*$ denotes the set of positive natural numbers. For any set $X$ and any mapping $T:X\to X$, we denote by  $Fix(T)$ the set of fixed points of $T$, that is, $Fix(T)=\{x\in X\mid T(x)=x\}$.

\section{A modified proof of a theorem of Browder}\label{modified}

Let us recall the following well-known strong convergence result due to Felix Browder.

\begin{theorem}[Browder]\label{browder} Let $X$ be a real Hilbert space and $U:X \to X$ a nonexpansive mapping. Assume that $C$ is a bounded closed convex subset of $X$, that $v_0 \in C$, and that $U$ maps $C$ into itself. For each natural number $n$, let 
\beq\label{Un-Browder}
U_n(x):=\left(1-\frac{1}{n+1}\right)U(x)+\frac{1}{n+1}\,v_0
\eeq
and consider $u_n$ to be the unique fixed point of this strict contraction. Then the sequence $(u_n)_{n\in \N}$ converges strongly to a fixed point of $U$ in $C$ {\em (}the closest one to $v_0${\em )}.
\end{theorem}

This theorem appeared in \cite{Browder(67)} (a map $U$ is said to be nonexpansive if  $\Vert U(x) - U(y)\Vert \leq \Vert x - y\Vert$,  for all $x,y \in X$). The purpose of this section is to adapt Browder's proof so that it can be  formalized in a theory $\Br^+$ for which a proof mining metatheorem applies. We present this theory  in Subsections~\ref{formal_subsection} and \ref{metatheorem_subsection}.  The highlight -- as described in the introduction -- is that $\Br^+$ postulates a form of Heine/Borel compactness for the strong topology of the Hilbert space $X$. The metatheorem is stated at  the end of Subsection~\ref{metatheorem_subsection}. We show in Subsection~\ref{adapted} that a modification of Browder's proof is formalizable in the theory $\Br^+$. Meanwhile, in the following, we review in broad lines Browder's argument and show how the sequential weak compactness argument of Browder can be replaced by an application of Heine/Borel compactness.

Browder starts his proof of the above theorem by showing that the set $F:=Fix(U)$ of fixed points of $U$ is nonempty, convex and closed. The arguments for the closedness and convexity of $F$ are simple. However, in order to argue that there is a fixed point, Browder refers to papers that rely on Zorn's lemma. Next, Browder's argument invokes Hilbert's projection theorem to justify the existence of a (unique) point of $F$ closest to $v_0$. The projection theorem can be proved in the following way. One considers $\lambda := \inf_{x\in F} \Vert x - v_0\Vert$. By definition, $$\forall k \exists x\in F  \left(\Vert x-v_0 \Vert \leq \lambda + \frac{1}{k+1}\right).$$ We can frame the above claim differently, and sidestep the existence of the infimum: 

\begin{equation}\label{near_closest}
\forall k \exists x\in C \left(U(x) = x \wedge \forall y\in C \left(U(y) = y \to \Vert x -v_0 \Vert 
< \Vert y - v_0 \Vert + \frac{1}{k+1}\right)\right).
\end{equation}

The projection argument proceeds by taking a sequence $(x_k)_{k\in \N}$ of fixed points of $U$ in $C$ such that, for all $k\in \N$, 
$$\forall y\in C\left(U(y) = y \to \Vert x_k - v_0 \Vert \leq \Vert y - v_0 \Vert + \frac{1}{k+1}\right).$$ One can show that $(x_k)$ is a Cauchy sequence and, hence, that it converges (to the point of $F$ closest to $v_0$). However, in order to obtain the sequence $(x_k)$, one needs a strong principle of choice (see the second part of the introduction). 

In his mining of Browder's theorem, Kohlenbach made the {\em crucial} observation that (\ref{near_closest}) is already enough to carry on with Browder's argument. Therefore, the mentioned application of choice is not needed. In order to continue Browder's argument, we need two technical facts:

\begin{enumerate}
\item[(I)] $\forall k \exists x\in C \, \left(U(x) = x \wedge \forall y\in C \, 
\left(U(y) = y \to \langle x-v_0, x-y \rangle < \frac{1}{k+1}\right)\right)$;
\item[(II)] $\forall n \forall x\in C \, \left(U(x) = x \to \Vert u_n - x \Vert^2 \leq 
\langle x - v_0, x-u_n \rangle \right)$.
\end{enumerate}

Both facts are essentially argued in Browder's paper. Fact (I) uses heavily the convexity of $C$ together with the projection result. Fact (II) is the combinatorial core of Browder's argument: its proof relies on very simple computations.

We can now prove that the sequence $(u_n)$ is a Cauchy sequence. To prove this, let $k  \in \N$ be given. By (I), take $\tilde{x} \in C$ such that $U(\tilde{x}) = \tilde{x}$ and 

\begin{equation}\label{adapted_a}
\forall y\in C \left(U(y) = y \to \langle \tilde{x} - v_0,\tilde{x} - y \rangle < \frac{1}{k+1}\right).
\end{equation}

By (II), it is enough to show that

\begin{equation}\label{enough}
\exists n \,\forall i\geq n \,\left(\langle \tilde{x}-v_0,\tilde{x}-u_i\rangle < \frac{1}{k+1}\right).
\end{equation}

Assume not. Then

\begin{equation}\label{assumption} 
\forall n \,\exists i\geq n \left(\langle \tilde{x}-v_0, \tilde{x}-u_i \rangle \geq \frac{1}{k+1}\right).
\end{equation}

Take $(v_n)$ a subsequence of $(u_n)$ such that 

\begin{equation*}
\forall n \left(\langle \tilde{x}-v_0, \tilde{x}-v_n\rangle \geq \frac{1}{k+1}\right)
\end{equation*}

At this point, we invoke a sequential weak compactness argument. Take $(w_n)$ a subsequence of $(v_n)$ weakly converging to a certain point $y\in C$. It is  easy to see 

\begin{equation}\label{quasi_fixed}
\forall n \left( \Vert U(u_n) - u_n \Vert \leq \frac{b}{n+1} \right),
\end{equation}

where $b$ is an upper bound of the diameter of $C$. Using this fact, it can be shown that $U(y) = y$.  By weak convergence, we get $\langle \tilde{x}-v_0,\tilde{x} -y\rangle \geq \frac{1}{k+1}$. This contradicts (\ref{adapted_a}).

At this juncture, we show how to replace the above sequential weak compactness argument by a Heine/Borel argument. We need to argue (\ref{enough}). By (\ref{adapted_a}), 
$$\forall y\in C\left( \forall m\in \N \left( \Vert U(y) - y \Vert \leq \frac{1}{m+1} \right) \to \langle \tilde{x}-v_0, \tilde{x}-y \rangle < \frac{1}{k+1} \right).$$ Hence, $C \subseteq \bigcup_m \Omega_m$, where $$\Omega_m := \left\{y\in X: \Vert U(y) - y \Vert > \frac{1}{m+1}\right\} \cup 
\left\{y \in X: \langle \tilde{x}-v_0, \tilde{x} - y \rangle < \frac{1}{k+1} \right\}.$$ By Heine/Borel compactness, there is $\ell \in \N$ such that $C \subseteq \Omega_\ell$ (note that the sequence of the $\Omega$s is increasing). Therefore $$\forall y\in C \left( \Vert U(y) - y \Vert \leq \frac{1}{\ell +1} \to \langle \tilde{x}-v_0, \tilde{x}-y \rangle < \frac{1}{k+1} \right).$$ 

Using (\ref{quasi_fixed}), it is clear that (\ref{enough}) follows.

\subsection{The formal theory}\label{formal_subsection}

In Section \ref{metatheorem_subsection}, we state a logical metatheorem of the kind discussed in the introduction. In order to do this in a rigorous manner, we need to describe an appropriate formal language, as well as to formulate appropriate theories. The present section is one of logical and technical flavor, its objective being the description of a formal theory $\Br$ adequate to formulate Browder's theorem.

The formal language of $\Br$ is the language of finite types with two base types: the base type 0 of 
the natural numbers and the (abstract) base type $X$. With some minor differences (discussed below), 
we follow the framework of chapter 17 of \cite{Kohlenbach(08)} (see also \cite{Kohlenbach(05)} and 
\cite{GerhardyKohlenbach(08)}). The language has only one relation symbol, namely the equality 
symbol $=_0$ of type 0. It includes the usual constants associated with logic (combinators, 
extended to the new types) and with arithmetic (zero, sucessor and the primitive recursive 
functionals in the sense of G\"odel, extended to the new types). For the base type $X$, 
there are some constants for inner product spaces and some {\em ad hoc} constants specially 
introduced to analyze Browder's strong convergence result. We find 
that the introduction of these {\em ad hoc} constants (and associated axioms) is very convenient 
because their presence greatly simplifies matters.

The proof of the metatheorem of Subsection~\ref{metatheorem_subsection} relies on the 
so-called bounded functional interpretation adapted to the new situation, with a base abstract 
type. So far, there is only one place in the literature where the bounded functional interpretation, 
extended to an abstract type, has been treated, viz. in the doctoral dissertation of Patr\'{\i}cia 
Engr\'acia \cite{Engracia(09)}. As it is characteristic of bounded functional interpretations, 
an intensional (i.e., rule-governed) majorizability relation plays a crucial role (as well as the notion of intensional bounded formula). In this paper, 
we have managed to avoid speaking of this intensional relation. The cost of the simplification 
is that the statement of the metatheorem in Subsection~\ref{metatheorem_subsection} is unduly 
restricted and not formulated in its proper natural setting. However, the restricted metatheorem 
is enough for our present purposes. Another cost is that we will not be able to prove the metatheorem 
in this paper. The proof requires the introduction of the full apparatus (or something close enough) 
and that would make the paper quite long. A proof can be obtained by adapting the proof of 
Theorem 35 of \cite{Engracia(09)}. Unfortunately, intensionality issues cannot be avoided 
altogether. A few of these issues do necessarily arise. We deal with these issues in a case-by-case 
basis (instead of uniformly, as it is done by the bounded functional interpretation). 

In order to be brief and follow familiar usage, we choose to rely on the available established 
literature as much as possible. Accordingly, we lean on Section 17.3 of \cite{Kohlenbach(08)} and 
treat inner product spaces as the special case of normed spaces in which the parallelogram law holds. 
There are vector space constants $0_X$, $+_X$, $-_X$ and $\cdot_X$ of types (respectively) 
$X$, $X \to (X \to X)$, $X \to X$ and $1 \to (X \to X)$.  They stand (respectively) for the zero 
vector and the vector operations of sum, symmetric of a vector and scalar multiplication. The 
functional for scalar multiplication should be interpreted in the following way: given a 
function $\gamma: \N \to \N$ and $x \in X$, $\gamma \cdot_X x$ is the scalar multiplication 
of the real $\gamma_\R$ with the vector $x$. Here, $\gamma_\R$ is the representation of a 
real associated -- in a primitive recursive way (in the sense of Kleene) -- to the number theoretic 
function $\gamma$. This can be done adequately in several ways. Kohlenbach uses in \cite{Kohlenbach(08)} 
a representation based on Cauchy sequences. Engr\'acia uses the signed digit representation in \cite{Engracia(09)}. 
We opt for the latter representation. Be that as it may, it is important to 
note that the (defined) relations of $=_\R$ and $\leq_\R$ between (representations of) real 
numbers are given by $\Pi^0_1$-formulas, and strict equality $<_\R$ is given by a $\Sigma^0_1$-formula. 
Furthermore, there is a constant $\Vert \!\cdot\! \Vert$ of type $X \to 1$ that stands for the norm. 
We assume that $\Vert x \Vert$ is always a type 1 functional of the form $\gamma_\R$.

There are also some {\em ad hoc} constants related to the statement of Browder's strong convergence result. 
We include a constant $v_0$ of type $X$ for the given point in Browder's theorem, a constant $C$ of 
type $X \to 0$ for the characteristic function of the bounded closed convex subset of $X$ and a 
constant $b$ of type 0 for a positive bound on the diameter of $C$. Finally, we also have a 
constant $u$ of type $0 \to X$ whose intended meaning is that $u(n)$ gives the unique fixed point of 
the contraction $U_n$. We write $x \in C$ instead of $C(x)=_0 0$ and $u_n$ instead of $u(n)$.

Let us introduce the theory $\Br$, framed in the above described language. Equality is treated  as in chapters 3 and 17 of \cite{Kohlenbach(08)}, with Spector's weak extensionality rule. The axioms related to the combinators and the arithmetical constants are as usual  (the scheme of induction is unrestricted).  Note that equality between elements of $X$,  written $x=_X y$, is a defined notion and stands for $\Vert x-y \Vert =_\R 0$ (equality  between elements of $X$ is a $\Pi^0_1$-notion). We do not have axioms stating that $=_X$ is  a congruence for the relevant notions, nor even that it is an equivalence relation (the  direct postulation of transitivity or congruence is not given by universal statements).  However, with a careful universal axiomatization of (real) normed vector spaces, it can be  proved that $=_X$ is indeed an equivalence relation and that it is congruent with respect to  the normed vector space notions (see Section 17.3 of \cite{Kohlenbach(08)} for details). It is  also congruent with respect to $U$, but {\em not} with respect to $C$. In fact,  $$\forall x^X,y^X (x=_X y \wedge x\in C \to y\in C),$$ is {\em not} provable in our theory (see Subsection~\ref{nonextensionality}).

To simplify matters and notation, we have a special axiom for the norm: 
$\forall x^X \forall n^0 (\Vert x \Vert (n) =_0 \Vert x \Vert_\R(n))$. This axiom says that 
the type 1 functional $\Vert x \Vert$ is always a representation of a real number. The inner product 
functional $\langle \,\, ,\,\rangle$ of type $X \to (X \to 1)$ is defined by 
$$\langle x, y\rangle := \frac{1}{4}\, (\Vert x + y\Vert^2 - \Vert x - y \Vert^2)$$ 
(we are using some liberty in the notation above -- omitting subscripts -- and will take such 
liberties whenever we find convenient). The above mentioned careful axiomatization of (real) 
normed vector spaces together with the axiom giving the parallelogram law 
$$\forall x^X,y^X (\Vert x + y\Vert^2 + \Vert x - y \Vert^2 = 
2 (\Vert x\Vert^2 + \Vert y\Vert^2))$$ entails the usual properties of the inner product.

We now describe the {\em ad hoc} axioms of $\Br$. Firstly, the axioms for $v_0$ and $C$:

\begin{enumerate}
\item[] $v_0 \in C$
\item[] $\forall x,y \in C \; (\Vert x-y\Vert \leq b_\R)$
\item[] $\forall x,y \in C \; \forall \gamma^1 \; ((1-\gamma_{[0,1]}) \cdot x + \gamma_{[0,1]} \cdot y  \in C)$
\end{enumerate}

We need to make several comments. The second axiom is an abbreviation of 
$$\forall x^X,y^X (C(x)=_0 0 \wedge C(y) =_0 0 \to \Vert x-y\Vert \leq_\R b_\R),$$ where $b_\R$ 
(of type 1) is a given representation of the real associated with $b^0$ (note the benign duplicity 
of the notation $b_\R$). We classify a quantification of the form $\forall x\in C \, (\ldots)$ 
as a bounded quantification. Dually, quantifications of the form $\exists x \in C\, (\ldots)$ 
are also classified as bounded.

In the above third axiom, $\gamma_{[0,1]}$ is $\max_\R (0_\R, \min (1_\R, \gamma_\R))$, where the notation is self-explanatory (the functionals $\max_\R$ and $\min_\R$ can be appropriately defined). However, it is crucial for our treatment that there exists a fixed functional $l^1$ such that $\forall \gamma^1 \forall n^0(\gamma_{[0,1]} (n) \leq l(n))$. This is the case with the signed digit representation: $l$ can be taken to be $\lambda n.5$ (cf. \cite{Engracia(09)}). We write the third axiom in a more readable way:

\begin{enumerate}
\item[] $\forall x,y \in C \; \forall \gamma \in [0,1] \; ((1-\gamma) \cdot x + \gamma \cdot y  \in C)$
\end{enumerate}

In general, a quantification of the form $\forall  \gamma^1 (\ldots \gamma_{[0,1]} \ldots)$ is 
written $\forall \gamma \in [0,1]\, (\ldots \gamma \ldots)$. Note that, due to lack of extensionality, 
this quantification is not always equivalent to 
$\forall \gamma^1 (0_\R \leq_\R \gamma_\R \leq_\R 1_\R \to (\ldots \gamma_\R \ldots))$. 
The latter quantification implies the former, but not vice-versa. We also classify quantifications 
of the form $\forall \gamma \in [0,1] \, (\ldots)$ or (dually) $\exists \gamma \in [0,1] \, (\ldots)$ 
as bounded quantifications.

There are also {\em ad hoc} axioms of $\Br$ for $U$. They are the following two axioms:

\begin{enumerate}
\item[] $\forall x \in C \; (U(x) \in C)$
\item[] $\forall x,y \in X \; (\Vert U(x) - U(y) \Vert \leq \Vert x - y \Vert)$
\end{enumerate}

Finally, there are two {\em ad hoc} axioms for the constant $u$ of type $0 \to X$: 

\begin{enumerate}
\item[] $\forall n \, (u_n \in C)$
\item[] $\forall n \left( \left(1 - \frac{1}{n+1} \right) U(u_n) + \frac{1}{n+1}\, v_0 = u_n \right)$
\end{enumerate}

The last axiom says that, for each natural number $n$, $u_n$ is the fixed point of $U_n$.

We are done describing the formal theory $\Br$.

\subsection{Brief semantical considerations}\label{semantics}

$\Br$ does not postulate the Cauchy completeness of the space $X$, nor the closedness of $C$. These are not universal properties and, therefore, not suitable for a proof mining metatheorem. This is characteristic of proof mining studies. However, the two (universal) axioms regarding the constant $u^{0 \to X}$ have a tinge of completeness. For each natural number $n$, $U_n$ is a strict contraction (with contraction constant $1-\frac{1}{n+1}$). A (unique) fixed point of $U_n$ is only guaranteed to exist by Cauchy completeness (Banach fixed point theorem). Moreover, the fixed point is only guaranteed to be in $C$ provided that $C$ is closed. At the cost of some complications, it would have been possible to work without the constant $u$ and its corresponding two axioms. We, nevertheless, opted for our present simpler treatment.

Let $X$ be a (real) Hilbert space, $C$ a closed convex subset with diameter bounded by a positive natural number $b$, $v_0 \in C$ and $U: X \to X$ a nonexpansive function that maps $C$ into itself. To each such quintuple $(X, C, b, v_0, U)$, we can associate a model of $\Br$. The base types 0 and  $X$ are interpreted by the natural numbers and by the given Hilbert space $X$, respectively. The remaining finite types are interpreted by the full set-theoretic structure over the base types 0 and $X$. With the exception of the norm, the interpretations of the constants are completely natural (as discussed, each $u_n$ is interpreted as the unique fixed point of the map $U_n$). In the case of the norm, a selection of a (signed digit) representative for each real number $\Vert x \Vert$ has to be made. Any selection will do for our purposes. 

In the metatheorem of the next subsection, a certain conclusion is provable in $\Br$. Therefore, the conclusion is true for the structures associated to the quintuples $(X, C, b, v_0, U)$ described in the previous paragraph.

\subsection{The first metatheorem}\label{metatheorem_subsection}

In this subsection, we add to the theory $\Br$ a principle of choice and two ``Heine/Borel principles." 
They are ``characteristic principles" of the bounded functional interpretation 
(cf. \cite{FerreiraOliva(05)}, \cite{Ferreira(09)} and \cite{Engracia(09)}). These characteristic 
principles trivialize under the bounded functional interpretation (in the same sense that, 
for instance, Markov's principle trivializes under G\"odel's {\em dialectica} interpretation for Heyting arithmetic). 
In order to formulate these principles, we need to introduce some simple notions.

A formula of the language of $\Br$ is called {\em bounded} if it can be obtained from atomic formulas using propositional connectives and bounded quantifications, i.e., quantifications of the form $\forall k\leq_0 n \, (\ldots)$, $\exists k\leq_0 n \, (\ldots)$, $\forall x\in C \, (\ldots)$, $\exists x\in C \, (\ldots)$, $\forall \gamma \in [0,1] \, (\ldots)$ or $\exists \gamma \in [0,1] \, (\ldots)$. A $\Sigma$-formula is a formula of the form $\exists n^0 B(n)$, where $B$ is a bounded formula. $\Pi$-formulas are defined dually. Given a type $1$ functional $f$, $\tilde{f}$ denotes the type $1$ functional given by $$\tilde{f}(n) := \max_{k\leq n} f(k)$$

The {\em bounded choice principle}, dubbed $\bAC$, is 
$$\forall n^0 \exists k^0 A(n,k) \to \exists f^1 \forall n \exists k\leq \tilde{f}(n) \, A(n,k),$$ 
where $A$ is a $\Sigma$-formula (possibly with parameters).

The {\em bounded collection principles}, dubbed $\BC$, are the following principles: $$\forall i\leq_0 m \exists n^0 A(i,n) \, \to \, \exists n \forall i\leq m \exists k\leq n \, A(i,k),$$ $$\forall \gamma \in [0,1] \, \exists n^0 A(\gamma,n) \, \to \, \exists n \forall \gamma \in [0,1] \, \exists k\leq n \, A(\gamma,k)$$ and $$\forall x\in C \exists n^0 A(x,n) \, \to \, \exists n \forall x\in C \exists k\leq n \, A(x,k),$$ where the $A$'s are $\Sigma$-formulas (possibly with parameters). The first principle is, 
of course, provable by induction. We nevertheless include it, in order to display theoretical uniformity. The last two principles are ``Heine/Borel" like. For instance, the third principle says that if $C \subseteq \bigcup_n \{x\in X: A(x,n)\}$, then $C$ is already covered by finitely many of the sets $\{x\in X: A(x,n)\}$. Note that this ``Heine/Borel" principle is restricted to countable coverings.

In the presence of the principles $\BC$, it is clear that the class of $\Sigma$-formulas is 
closed under bounded quantifications. This fact will be used many times in the sequel without mention.

\begin{definition} 
The theory $\Br^+$ is the theory $\Br$ together with the principles $\bAC$ and $\BC$. 
\end{definition}

Our first proof mining metatheorem is the following:

\begin{metatheorem} \label{first_metatheorem}
Suppose that the theory $\Br^+$ proves the sentence $\forall k^0 \forall f^1 \exists n^0 A(k,\tilde{f},n)$, 
where $A$ is a $\Sigma$-formula. Then there is a closed functional 
$\phi$ of type $0\to (1 \to 0)$ 
such that the theory $\Br$ proves 
$$\forall k \forall f \exists n\leq \phi(k,f) \, A(k,\tilde{f},n).$$ 
\end{metatheorem}

This is partly a conservation result. Clearly, if $\Br^+ \vdash \forall k \forall f \exists n A(k,\tilde{f},n)$, then $\Br \vdash \forall k \forall f \exists n A(k,\tilde{f},n)$. If the variable $f$ is absent, we get in particular $\Pi^0_2$-conservation. Note also that, when the $\Pi^0_2$-sentence is just $0=1$, we obtain the relative consistency of the theory $\Br^+$ with respect to the theory $\Br$. This observation shows the nontrivial fact that the theory $\Br^+$ is consistent (note that there is no obvious model of $\Br^+$).

Additionally, the above metatheorem is an extraction result in the sense that it extracts a bound $\phi$ from a certain given (formal) proof. The bound is extracted constructively. This is a consequence of the proof of the metatheorem. Given a formal derivation $\sigma_1$, $\sigma_2$, \dots, $\sigma_n$ in $\Br^+$, one effectively associates with it a sequence of formulas $\sigma_1^\bfi$, $\sigma_2^\bfi$, \dots, $\sigma_n^\bfi$ (given by the bounded functional interpretation) and a sequence of functionals $\phi_1$, $\phi_2$, \ldots, $\phi_n$ such that each $\phi_i$ bounds the ``existential witnesses" of $\sigma_i^\bfi$, provably so in $\Br$. In the minings of the theorems of this paper -- given that their ordinary proofs have a simple {\em logical} complexity  -- it is sufficient to use only the above characteristic principles to effect a $\bfi$-transformation and obtain formulas of the form $\forall k \forall f \exists n A(k,\tilde{f},n)$, with $A$ a $\Sigma$-formula 
(the full apparatus of the bounded functional interpretation is not needed). We call this a {\em quantitative form} of the given formula. One then tries to construct bounds $\phi$ as in the theorem. This {\em modus faciendi} is well illustrated in the analysis of the projection argument in Section~\ref{projection}.

\subsection{The modified proof in the formal theory}\label{adapted}

The main point of this section is to show that Browder's proof can be modified so that the  following holds:

\begin{theorem}\label{adapted_thm} 
The theory $\Br^+$ proves 
$$\forall k \,\exists n \, \forall m \, \forall i,j \in [n,m] \left( \Vert u_i - u_j\Vert <_\R \frac{1}{k+1}\right).$$ 
\end{theorem}

The metatheorem of the previous section cannot be applied directly to the conclusion above, because the latter does not have the right logical form. As discussed in the introduction, we  consider the metastable version of the conclusion:

\begin{corollary}\label{prediction} 
There is a closed functional $\phi$ of type $0\to (1\to 0)$ such that the theory $\Br$ proves 
$$\forall k^0 \forall f^1 \exists N\leq \phi(k,f) \, \forall i,j  \in [N, \tilde{f}(N)] \left( \Vert u_i - u_j\Vert <_\R \frac{1}{k+1} \right).$$ 
\end{corollary}

\begin{proof} By the above theorem, the theory $\Br^+$ proves  $\forall k \exists n \forall m \, \forall i,j \in [n,m] \,( \Vert u_i - u_j\Vert <_\R \frac{1}{k+1})$. 
It logically follows that 
$$\forall k \forall f \exists N  \, \forall i,j \in [N,\tilde{f}(N)] \left( \Vert u_i - u_j\Vert <_\R \frac{1}{k+1} \right)$$
Note that the formula after the quantification ``$\exists N$" is equivalent to a $\Sigma$-formula.  Now, just apply the metatheorem of the previous subsection. 
\end{proof}

The corollary predicts the existence of a closed functional $\phi$ as above. An explicit bounding functional $\phi$ is computed in Theorem \ref{app:quant-browder}. This is what a proof mining study amounts to. 

Let us now prove Theorem \ref{adapted_thm}. We follow the review of Browder's argument in the opening of this section. In $\Br^+$ we cannot speak literally of sets, and we also do not need the closedness of $F:=Fix(U)$. The convexity of $F$ can be stated in the formal language as 

\begin{equation}\label{convexity}
\forall x,y \in C\, \forall \gamma \in [0,1] \,\big(U(x) = x \wedge U(y) = y \to U((1-\gamma) x + \gamma y) = (1-\gamma) x + \gamma y \big),
\end{equation}

and its proof can be readily formalized in $\Br$. The claim that $\exists x\in C \, \big(U(x) = x \big)$ is quite another matter. For strict proof mining purposes, there is however a very simple way of dealing with this problem. Just postulate the existence of a fixed point! Formally, this means extending the language with a constant $c$ of type $X$ and accepting the axiom $U(c) = c$. This works because the axiom is universal and true. However, as it happens, it is very easy to prove in $\Br^+$ the existence of a fixed point. This  follows immediately from $\BC$ using (\ref{quasi_fixed}). We can apply (the contrapositive of) $\BC$ in the form $$\forall n \exists x\in C \forall k \leq n \left( \Vert U(x) - x \Vert \leq \frac{1}{k+1} \right) \to \exists x \in C \, \forall n  \left( \Vert U(x) - x \Vert \leq \frac{1}{n+1}\right)$$ in order to conclude the existence of fixed points.

As a matter of fact, the above argument can be seen as an application of Heine/Borel compactness for the strong topology. Suppose that there is no such fixed point. For each $n\in \N$, let $\Gamma_n$ be the open set $\{x\in X: \| U(x) -x \| > \frac{1}{n+1}\}$. By the supposition, $C \subseteq \bigcup_n \Gamma_n$. Therefore, by Heine/Borel compactness, there is $\ell \in \N$ such that $C \subseteq \Gamma_\ell$ (note that the sequences of $\Gamma$s is increasing). Then, obviously, for all $x\in C$, $\| U(x) - x\| > \frac{1}{\ell +1}$. This contradicts (\ref{quasi_fixed}).

Next, Browder's argument invokes Hilbert's projection theorem. As we saw, it is enough to show (\ref{near_closest}). This fact can easily be proved in $\Br^+$ by contradiction. To see this, suppose that there is $k_0$ such that $$\forall x\in C \left( U(x) = x \to \exists y \in C \left(U(y) = y \wedge \Vert x - v_0 \Vert  
\geq \Vert y-v_0 \Vert + \frac{1}{k_0+1}\right)\right).$$

Let $x_0 \in C$ be a fixed point of $U$. A simple inductive argument shows that $$\forall n \exists y\in C \left( U(y) = y \wedge \Vert x_0 -v_0 \Vert \geq \Vert y - v_0 \Vert + \frac{n}{k_0+1} \right).$$ This is a contradiction.

Browder's argument, as we saw in the beginning of this section, needs the two technical facts (I) and (II). These two facts can be proven in $\Br$ without much ado. Finally, as discussed, we can use Heine/Borel compactness in order to prove that the sequence $(u_n)$ is a Cauchy sequence. The crucial passage is from $$\forall y\in C \, \exists  m^0 \left( \Vert U(y) - y \Vert \leq \frac{1}{m+1} \to \langle \tilde{x}-v_0, \tilde{x}-y \rangle < \frac{1}{k+1} \right)$$ to $$\exists \ell \, \forall y\in C \exists  m\leq \ell \, \left( \Vert U(y) - y \Vert \leq \frac{1}{m+1} \to \langle \tilde{x}-v_0, \tilde{x}-y \rangle < \frac{1}{k+1} \right).$$ This follows from an application of the bounded collection principle $\BC$.

\subsection{Observations}

Before starting a new section, we make four observations:

\subsubsection{The uniform boundedness principle}

We have presented an argument that circumvents the application of weak sequential compactness. Our argument uses the characteristic principles of the bounded functional interpretation (specially bounded collection) but the argument can also be made within the framework of the monotone functional interpretation by appealing to the so-called generalized uniform boundedness principle $\exists${\sf -UB}$^X$ (see \cite{Kohlenbach(06)} or Sections 17.7 and 17.8 of \cite{Kohlenbach(08)}). Whereas the bounded functional interpretation {\em trivializes} the use of Heine/Borel compactness, the monotone functional interpretation interprets it with the aid of a postulate which is true in Bezem's structure of the strongly majorizable functionals. With the latter technique, one does {\em not} obtain the provability of the conclusion of Corollary \ref{prediction} in the theory $\Br$ (no conservation result is obtained), but only the set-theoretic truth of the conclusion. Of course, the latter is already sufficient for proof mining purposes. It remains to be seen whether the bounded functional  interpretation, with its bounded collection principles, has applications that cannot be  obtained using $\exists${\sf -UB}$^X$ instead.

\subsubsection{On getting to the truth with false principles}

The proof of Theorem \ref{adapted_thm} gives rise to a new proof of Theorem \ref{browder}. 
In effect, given $X$, $C$, $v_0$ and $U$ as in Browder's theorem, and given $b\in \N$ a positive 
bound for the diameter of $C$, the conclusion of Corollary \ref{prediction} is true in the 
structure $(X, C, b, v_0, U)$. Using countable choice in the real world, that conclusion 
implies the Cauchyness -- and, hence, the convergence -- of the sequence $(u_n)$ in $X$. 
The object lesson is that the use of the false Heine/Borel compactness principle in the context of 
$\Br^+$ is a perfectly good way of proving the convergence of sequences in Hilbert spaces.

\subsubsection{The non-extensionality of the convex set {\it C}}\label{nonextensionality}

In Subsection~\ref{formal_subsection}, we said that $C$ is not extensional in the sense that  
$\forall x^X,y^X (x=_X y \wedge x\in C \to y\in C)$ is not provable in $\Br$. In fact, it is 
not provable in $\Br^+$. For, suppose it is. By instanciating the variable $x$ by $0_X$ and 
the variable $y$ by $-y$, we get $$\Br^+ \vdash \forall y^X (y=_X 0_X \wedge 0_X \in C \to -y \in C).$$ 
In particular $$\Br^+ \vdash \forall y\in C \left( \forall n \left(\Vert y \Vert \leq \frac{1}{n+1}\right) 
\wedge 0_X \in C \to -y \in C\right)$$ and, hence, 
$$\Br^+ \vdash \forall y\in C \, \exists n \left(\Vert y \Vert \leq \frac{1}{n+1} \wedge 0_X \in C 
\to -y \in C\right).$$ 
Applying the bounded collection principle $\BC$, we get 
$$\Br^+ \vdash \exists n \forall y\in C \left(\Vert y \Vert \leq \frac{1}{n+1} 
\wedge 0_X \in C \to -y \in C\right).$$ 
By the first metatheorem, there is a concrete number $n_b$ (which depends only on a bound $b$ for 
the diameter of $C$) such that 
$$\Br \vdash \forall y\in C \left(\Vert y \Vert \leq \frac{1}{n_b+1} \wedge 0_X \in C \to -y \in C\right).$$ 
Therefore, the sentence after the provability sign is true in the structure 
$(\R, [0,1], 1, 0, {\rm id})$. We get 
$$\forall y \in [0,1] \left( \vert y \vert \leq \frac{1}{n_1+1} \to -y \in [0,1]\right).$$ 
This is false.

\subsubsection{Weakening the base theory}\label{weakening}

Instead of the theory $\Br$, one can also consider its fragment $\Br\!\!\upharpoonright$, 
where we have only the recursor $R_0$ for type-0 recursion, and induction is restricted to 
the {\em scheme of bounded induction}: 
$$A(0) \wedge \forall n^0 (A(n) \to A(n+1)) \to \forall n A(n),$$ 
where $A$ is a bounded formula, possibly with parameters. Note that the arithmetical functionals 
of the language of $\Br\!\!\upharpoonright$ are the so-called Kleene's primitive recursive 
functionals of finite type (see \cite{Kohlenbach(08)}). The theory $\Br\!\!\upharpoonright^+$ is 
the theory $\Br\!\!\upharpoonright$ together with the principles $\bAC$ and $\BC$.

\begin{metatheorem_2}\label{second_metatheorem} 
Suppose that the theory $\Br\!\!\upharpoonright^+$ proves the sentence 
$\forall k^0 \forall f^1 \exists n^0 A(k,\tilde{f},n)$, where $A$ is a $\Sigma$-formula. 
Then there is a closed functional $\phi$ of type $0\to (1 \to 0)$, primitive recursive in the 
sense of Kleene, such that the theory $\Br\!\!\upharpoonright$ proves 
$$\forall k \forall f \exists n\leq \phi(k,f) \, A(k,\tilde{f},n).$$ 
\end{metatheorem_2}

This theorem is proved like the First Metatheorem of Subsection~\ref{metatheorem_subsection}. 
The scheme of bounded induction does not pose a problem for the bounded functional interpretation 
because it is equivalent to a scheme of $\Pi$-formulas: 
$$\forall n^0 \big(A(0) \wedge \forall r<n (A(r) \to A(r+1)) \to A(n)\big),$$ where 
$A$ is a bounded formula, possibly with parameters (not even the recursor $R_0$ is necessary 
to interpret the scheme of bounded induction). However, as it is easy to argue, the presence 
of the recursor $R_0$ (and of $\bAC$) permits to lift induction to $\Sigma$-formulas. 
A well-known trick of bounded arithmetic shows that one can also derive induction for $\Pi$-formulas.

It should be noticed that the three theorems mined in this paper have proofs that can be 
formalized within the restrited theory $\Br\!\!\upharpoonright^+$.

\section{The interpretation and mining of the projection argument}\label{projection}

As reported in Subsection~\ref{adapted}, the weaker projection statement (\ref{near_closest})  is already sufficient to carry on Browder's argument. In this section, we interpret and mine the proof of this projection statement using the bounded functional interpretation. We apply the methodology described at the end of Subsection~\ref{metatheorem_subsection}: we will try to rewrite (\ref{near_closest}) in a quantitative form $\forall k \forall f \exists N A(k,\tilde{f},n)$, with $A$ a $\Sigma$-formula, using the characteristic principles $\bAC$ and $\BC$. Well, (\ref{near_closest}) is equivalent to 
$$\forall k \exists x\in C \left[U(x) = x \wedge \forall y \in C\left( \forall N \left(\Vert U(y) - y \Vert \leq \frac{1}{N+1} \right)\to \Vert x -v_0 \Vert^2 < \Vert y-v_0 \Vert^2 + \frac{1}{k+1}\right)\right]$$ and, hence, equivalent to $$\forall k \exists x\in C \left(U(x) = x \wedge \forall y \in C \,\exists N \left(\Vert U(y) - y \Vert \leq 
\frac{1}{N+1} \to \Vert x -v_0 \Vert^2 < \Vert y-v_0 \Vert^2 + \frac{1}{k+1}\right)\right).$$ 

Since the formula after the quantifier `$\exists N$' is equivalent to a $\Sigma$-formula, by $\BC$ we easily get $$\forall k \exists x\in C \left(U(x) = x \wedge \exists N \forall y \in C \left(\Vert U(y) - y \Vert \leq \frac{1}{N+1} \to \Vert x -v_0 \Vert^2 < \Vert y-v_0 \Vert^2 + \frac{1}{k+1}\right)\right)$$ or, equivalently, $$\forall k \exists N \exists x^C \forall m \left(\Vert U(x) - x \Vert < \frac{1}{m+1} \, \wedge \, \forall y^C \left(\Vert U(y) - y \Vert \leq \frac{1}{N+1} \to \Vert x -v_0 \Vert^2 < \Vert y-v_0 \Vert^2 + \frac{1}{k+1}\right)\right).$$ (We are writing $x^C$ instead of $x \in C$ for typesetting reasons, in order to save space.) In turn, this obviously implies $$\forall k \exists N \forall m \exists x^C \left(\Vert U(x) - x \Vert < \frac{1}{m+1} \wedge \forall y^C \left(\Vert U(y) - y \Vert \leq \frac{1}{N+1} \to \Vert x -v_0 \Vert^2 < \Vert y-v_0 \Vert^2 + \frac{1}{k+1}\right)\right)$$ and, with more reason, $$\forall k \forall f \exists N \, \exists x^C \left(\Vert U(x) - x \Vert < \frac{1}{\tilde{f}(N)+1} \wedge \forall y^C \left(\Vert U(y) - y \Vert \leq \frac{1}{N+1} \to \Vert x -v_0 \Vert^2 < \Vert y-v_0 \Vert^2 + \frac{1}{k+1}\right)\right).$$ 

Given that the formula after the quantifier `$\exists N$' is equivalent to a $\Sigma$-formula, the bounded functional interpretation (see the metatheorem of Section~\ref{modified}) guarantees the existence of a bounding functional of $N$ in terms of $k$ and $\tilde{f}$. The obtaining of the bounding gives the quantitative version of the projection argument. (The reader should compare the end formula above with the corresponding formula given by the monotone functional interpretation in p. 2772 of \cite{Kohlenbach(11)}.)

The bounding property guaranteed by the metatheorem of Section~\ref{modified} is provable in $\Br$ but, following the usual procedure of proof mining studies, we will only worry next about the {\em truth} of the statement in appropriate structures. So, in the remainder of this section, $X$ is a pre-Hilbert space, $C$ is a convex subset of $X$ whose diameter is bounded by a positive natural number $b$, $U$ is a nonexpansive function that maps $C$ into $C$ and $v_0$ is a point of $C$. Since $(\ref{near_closest}$) was proved by induction, it is not a surprise that the bounding of $N$ is defined by recursion:

\begin{proposition}\label{proj1}
For any $k\in \N$ and $f:\N \to \N $, there is $N \in \N$ such that $N\leq (\tilde{f}+1)^{(r)}(0)$ and $$\exists x \in C \left( \|U(x)-x\| < \frac{1}{\tilde{f}(N)+1} \, \wedge \, \forall y \in C \left(\|U(y)-y \|\leq \frac{1}{N+1} \to \|x-v_0\|^2 < \|y-v_0\|^2 + \frac{1}{k+1}\right)\right),$$where $r:=b^2(k+1)$ and $(\tilde{f}+1)^{(r)}$ is the $r$-th fold composition of the function $\tilde{f} + 1$.
\end{proposition}

\begin{proof} It is convenient to prove instead the following slight modification of the above proposition: for any $k\in \N$ and $f:\N \to \N $, there is $N \in \N$ such that $N\leq \tilde{f}^{(r)}(0)$ and $$\exists x \in C \left( \|U(x)-x\| \leq \frac{1}{\tilde{f}(N)+1} \, \wedge \, \forall y \in C \left(\|U(y)-y \|\leq \frac{1}{N+1} \to \|x-v_0\|^2 < \|y-v_0\|^2 + \frac{1}{k+1}\right)\right),$$where $r:=b^2(k+1)$.

The difference is twofold. In the new version, the witness $N$ is taken to be bounded by the $r$-th fold composition of $\tilde{f}$, instead of $\tilde{f} +1$. This better bound is possible because we have relaxed the conclusion (replacing a $<$ by a $\leq$). Note that the above proposition is an immediate consequence of the new version by applying it to $f+1$.

Assume that the modified result is not true. Then there are $k\in \N $ and $f:\N \to \N $, such that for all $N\in \N$ with $N\leq \tilde{f}^{(r)}(0)$:

\begin{equation}\label{proj.eq1}
\forall x^C \left( \|U(x)-x\| \leq \frac{1}{\tilde{f}(N)+1} \to \exists y^C \left(\|U(y)-y\|\leq \frac{1}{N+1} \wedge \|y - v_0\|^2 \leq \|x-v_0\|^2-\frac{1}{k+1}\right)\right).
\end{equation}

First of all, note that the $i$-sequence given by the expression $\tilde{f}^i(0)$ is monotone (because $\tilde{f}$ is). We define a finite sequence $x_0$, $x_1$, \ldots, $x_r$, $x_{r+1}$ of elements of $C$ as follows:

\noindent $\underline{x_0}:$ \\ By (\ref{quasi_fixed}), let $x_0$ be such that

\begin{equation*}
\|U(x_0)- x_0\| \leq \frac{1}{\tilde{f}^{(r+1)}(0)+1}.
\end{equation*}

\noindent $\underline{x_{j+1} \mbox{, for } j\leq r}:$\\

Assume that we have $x_j$ such that $\|U(x_j)-x_j\| \leq \frac{1}{\tilde{f}^{(r-j+1)}(0)+1}$. By \eqref{proj.eq1} applied to $N=\tilde{f}^{(r-j)}(0)$ and to $x=x_j$, we conclude that there is $y\in C$ satisfying

\begin{equation*}
\|U(y)-y\|\leq \frac{1}{\tilde{f}^{(r-j)}(0)+1} \, \wedge \, \|y- v_0\|^2 \leq \|x_j - v_0\|^2-\frac{1}{k+1}.
\end{equation*}

Let $x_{j+1}$ be one such $y$.

By the definition, for all $j\leq r$,

\begin{equation*}
\|x_{j+1} - v_0\|^2\leq \|x_j - v_0\|^2 - \frac{1}{k+1},
\end{equation*}
\noindent which implies the contradiction 
\begin{equation*}
\|x_{r+1} -v_0\|^2\leq \|x_0-v_0\|^2-\frac{r+1}{k+1} \leq b^2 - \frac{b^2(k+1)+1}{k+1} < 0.
\end{equation*} 

\end{proof}

We have insisted upon the formulation of the above proposition with a $<$ instead of a $\leq$ because it is the former version, not the latter, which is the quantitative form of (\ref{near_closest}). In this paper, we make the point of being very clear about the theoretical background of the calculations. Nevertheless, for the strict purpose of getting bounds, we could have worked with the simpler bound (and this would reflect as well in the bounds computed in the remainder of this paper). As a matter of fact, the simpler bound can also be accounted theoretically but the explanation for this relies on intensional majorizability (cf. \cite{Pinto(19)}).

After this comment, let us proceed with the minings. Claim (I) of Section~\ref{modified} is a refinement of the projection result. Its bounded functional interpretation is similar to the interpretation of (\ref{near_closest}). It is:

$$\forall k \forall f \exists N \, \exists x^C \left(\Vert U(x) - x \Vert < \frac{1}{\tilde{f}(N)+1} 
\wedge \forall y^C \left(\Vert U(y) - y \Vert \leq \frac{1}{N+1} \to \langle x-v_0,x-y \rangle < \frac{1}{k+1}\right)\right).$$

The corresponding mining result is:

\begin{proposition}\label{proj3}
For any $k\in \N$ and $f:\N \to \N $, there is $N \in \N$ such that $N\leq 12b((\check{f}+1)^{(R)}(0)+1)^2$ and $$\exists x\in C \left( \|U(x)-x \| < \frac{1}{\tilde{f}(N)+1} \, \wedge \, \forall y\in C \left(\|U(y)-y\|\leq \frac{1}{N+1} \to \langle x-v_0,x-y\rangle < \frac{1}{k+1}\right)\right),$$ with $R:=b^4(k+1)^2+b^2$ and $\check{f}(m):=\max\{ \tilde{f}(12b(m+1)^2),\, 12b(m+1)^2 \}$.
\end{proposition}

As we will explain, this mining can be obtained from Proposition~\ref{proj1} and the following two estimates, essentially due to Kohlenbach in \cite{Kohlenbach(11)}:

\begin{lemma}\label{lem.conv}
For all $k\in \N $ and $x_1,x_2\in C$, $$\bigwedge_{j=1}^{2}\left( \|U(x_j) - x_j \|\leq \frac{1}{12b(k+1)^2} \right)\, \to \, \forall \gamma \in [0,1] \left(\|U(w_\gamma(x_1,x_2))-w_\gamma(x_1,x_2)\| < \frac{1}{k+1}\right).$$ 
\end{lemma}

In the above (and below), $w_\gamma(u,v):= (1 - \gamma) u+ \gamma v, \, \text{ for }\, \gamma \in [0,1]$.

\begin{lemma}\label{prop.lem}
For all $k\in \N $ and $x,y\in C$, $$\forall \gamma\in [0,1] \! \left(  \|x-v_0\|^2\leq \|w_\gamma(x,y)-v_0\|^2+\frac{1}{b^2(k+1)^2+1}\right) \, \to \, 
\langle x-v_0, x-y \rangle < \frac{1}{k+1}.$$
\end{lemma}

Lemma \ref{lem.conv} is nothing but the mining of (\ref{convexity}) using the bounded functional interpretation. The explanation for this is more straightforward than that of the  projection result (\ref{near_closest}). Let us do it. As usual, we apply the methodology of Subsection~\ref{metatheorem_subsection}. The above cited claim (\ref{convexity}) says that, for all $x_1,x_2 \in C$, $$\bigwedge_{j=1}^2 \forall m \left( \Vert U(x_j) - x_j \Vert \leq \frac{1}{m+1}\right) \to \forall \gamma\in [0,1] \, \forall k  \left(\Vert U(w_\gamma(x_1,x_2))-w_\gamma(x_1,x_2) \Vert < \frac{1}{k+1}\right).$$ By classical logic, this is equivalent to $$\forall k \exists m \left[\,\, \bigwedge_{j=1}^{2}\left( \|U(x_j) - x_j \|\leq \frac{1}{m+1} \right)\, \to \, \forall \gamma \in [0,1] \left(\|U(w_\gamma(x_1,x_2))-w_\gamma(x_1,x_2)\| < \frac{1}{k+1}\right)\right].$$ Given that $x_1$ and $x_2$ are arbitrary elements of $C$, two successive applications of $\BC$ yield, $$\forall k \exists m \forall x_1^C \forall x_2^C \left[\,\, \bigwedge_{j=1}^{2}\left( \|U(x_j) - x_j \|\leq \frac{1}{m+1} \right)\, \to \, \forall \gamma \in [0,1] \left(\|U(w_\gamma(x_1,x_2))-w_\gamma(x_1,x_2)\| < \frac{1}{k+1}\right)\right].$$ The formula between square parentheses is equivalent to a $\Sigma$-formula and, therefore, the metatheorem of Section~\ref{modified} predicts a bound for $m$ in terms of $k$. That bound was computed by Kohlenbach and is presented in Lemma~\ref{lem.conv}.

We have the following interim result:

\begin{corollary}\label{proj2}
For any $k\in \N$ and $f:\N \to \N $, there is $N \in \N$ such that $N\leq 12b((\check{f}+1)^{(r)}(0)+1)^2$ and there is $x\in C$ for which the following two properties hold: $$\|U(x)-x\| < \frac{1}{\tilde{f}(N)+1}$$ and $$ \forall y\in C \left(\|U(y)-y\| \leq \frac{1}{N+1} \to \forall \gamma \in [0,1] \left(\|x-v_0\|^2 < \|w_\gamma(x,y) - v_0\|^2 + \frac{1}{k+1}\right)\right),$$ with $r:=b^2(k+1)$ and $\check{f}(m):=\max\{ \tilde{f}(12b(m+1)^2),\, 12b(m+1)^2 \}$.
\end{corollary}

\begin{proof} Let $k$ and $f$ be given. By Proposition \ref{proj1}, there exist $x\in C$ and $N'\in \N$ such that $N' \leq (\check{f}+1)^{(r)}(0)$ with 

\begin{equation}\label{help_1}
\Vert U(x) - x \Vert < \frac{1}{\check{f}(N') +1} \wedge \forall y\in C \left( \Vert U(y) - y\Vert \leq 
\frac{1}{N'+1} \to \|x-v_0\|^2 < \|y-v_0\|^2 + \frac{1}{k+1}\right),
\end{equation} 

where $r = b^2 (k+1)$. Let $N : = 12b(N' +1)^2$. Clearly, $N \leq 12b ((\check{f}+1)^{(r)}(0) + 1)^2$. This entails that $$\Vert U(x) - x \Vert < \frac{1}{\tilde{f}(N) +1}$$ because $\tilde{f}(N) = \tilde{f}(12b(N' +1)^2) \leq \check{f}(N')$. Now, take $y\in C$ such that $\Vert U(y) - y\Vert \leq \frac{1}{N+1}$. Hence $\Vert U(y) - y\Vert \leq \frac{1}{12b(N' +1)^2}$. On the other hand, we also have $$\Vert U(x) -x \Vert < \frac{1}{\check{f}(N') +1} \leq \frac{1}{12b(N'+1)^2}.$$ By Lemma~\ref{lem.conv}, we get $\Vert U(w_\gamma(x,y)) - w_\gamma(x,y)\Vert \leq \frac{1}{N'+1}$. The result follows from the second conjunct of  (\ref{help_1}). 
\end{proof}

Lemma~\ref{prop.lem} is nothing but the mining of the following result: 
$$\forall \gamma \in [0,1]\! \left(  \|x-v_0\|^2\leq \|w_\gamma(x,y)-v_0\|^2\right) \, \to \, \langle x-v_0, x-y \rangle 
\leq 0.$$

This result is implicit in Browder's proof \cite{Browder(67)} and is needed to show (I) of Section~\ref{modified}. Proposition~\ref{proj3} is an immediate consequence of Lemma~\ref{prop.lem} and Corollary~\ref{proj2}:  just instantiate in the latter $k$ by $b^2(k+1)^2$.

\subsection{An observation concerning the mining of the projection argument}\label{proj_observation}

In rewriting (\ref{near_closest}) in a quantitative form, we seem to have lost the equivalence between (\ref{near_closest}) and the quantitative form itself. A similar situation already occurred in the proof of Corollary~\ref{prediction}. There is nothing wrong in proceeding in this manner (as long as the weakening is sufficient to carry out the mining analysis through the end). However, it so happens that the equivalence in these two cases has not been lost provided that one uses the characteristic principles. In the case of Corollary~\ref{prediction}, the equivalence obtains due to the bounded choice principle. In the case of the mining of the projection argument, one needs to argue that the logical transition from 

$$\forall k \exists N \exists x^C \forall m \left(\Vert U(x) - x \Vert < \frac{1}{m+1} \, \wedge \, \forall y^C \left(\Vert U(y) - y \Vert \leq \frac{1}{N+1} \to \Vert x -v_0 \Vert^2 < \Vert y-v_0 \Vert^2 + \frac{1}{k+1}\right)\right)$$ 
to 
$$\forall k \forall f \exists N \exists x^C \left(\Vert U(x) - x \Vert < \frac{1}{\tilde{f}(N)+1} \wedge \forall y^C \left(\Vert U(y) - y \Vert \leq \frac{1}{N+1} \to \Vert x -v_0 \Vert^2 < \Vert y-v_0 \Vert^2 + \frac{1}{k+1}\right)\right)$$ 
is, in fact, an equivalence. Well, the first formula above is equivalent to 
$$\forall k \exists N \exists x^C \forall m \left(\Vert U(x) - x \Vert \leq \frac{1}{m+1} \wedge \forall y^C \left(\Vert U(y) - y \Vert < \frac{1}{N+1} \to \Vert x -v_0 \Vert^2 \leq \Vert y-v_0 \Vert^2 + \frac{1}{k+1}\right)\right).$$ 
We made this move to ensure that the formula after the quantifier `$\forall m$' is equivalent to a $\Pi$-formula. Hence, by (the contrapositive of) $\BC$, this is equivalent to 
$$\forall k \exists N \forall m \exists x^C \left(\Vert U(x) - x \Vert \leq \frac{1}{m+1} \wedge \forall y^C \left(\Vert U(y) - y \Vert < \frac{1}{N+1} \to \Vert x -v_0 \Vert^2 \leq \Vert y-v_0 \Vert^2 + \frac{1}{k+1}\right)\right).$$
The formula after the quantifier `$\forall m$' is equivalent to a $\Pi$-formula and so, by (the contrapositive of) $\bAC$, we obtain
$$\forall k \forall f \exists N \exists x^C \left(\Vert U(x) - x \Vert \leq \frac{1}{\tilde{f}(N)+1} \wedge \forall y^C \left(\Vert U(y) - y \Vert < \frac{1}{N+1} \to \Vert x -v_0 \Vert^2 \leq \Vert y-v_0 \Vert^2 + \frac{1}{k+1}\right)\right)$$
This is equivalent to what we want. Note -- and this is interesting -- the equivalence for the projection case uses the set-theoretically false collection principles.

We had to perform a curious dance between strict and nonstrict inequalities in order to put the formulas in the right complexity class, but this is unavoidable without the intensional sign mentioned in Subsection~\ref{formal_subsection}.

\section{A useful general principle and its mining}\label{generalprinciple}

In this section, we isolate the technique that replaces the weak compactness argument by the 
Heine/Borel covering principle. We also obtain a quantitative theorem that can be used in 
a number of situations. The formal theory behind our result is the theory of bounded metric spaces. 
The language has two base types 0 and $X$, the latter for a (bounded) metric space. The situation 
is analogous to the language of the theory $\Br$ of the Subsection~\ref{formal_subsection}. 
In the case at hand, there are only two (mathematical) constants for the metric spaces: 
a constant $d$ of type $X \to (X \to 1)$ for the distance function, and a positive constant 
$b$ of type 0 that bounds the diameter of $X$. Equality $x =_X y$ between elements of the 
metric space is defined as $d(x,y) =_\R 0$. The theory $\Ms$ follows the same framework as before. 
We have the special (and simplifying) axiom $\forall x^X,y^X \forall n^0 (d(x,y)(n) =_0 d(x,y)_\R(n))$ 
and the following four (mathematical) axioms for bounded metric spaces:

\begin{enumerate}
\item[(i)] $\forall x^X (d(x,x) =_\R 0_\R)$
\item[(ii)] $\forall x^X,y^X (d(x,y) =_\R d(y,x))$
\item[(iii)] $\forall x^X,y^X,z^X (d(x,z) \leq_\R d(x,y) + d(y,z))$
\item[(iv)] $\forall x^X, y^X (d(x,y) \leq_\R b_\R)$
\end{enumerate} 
The notion of bounded quantification in the bounded metric setting is similar to the one given 
in Subsection~\ref{metatheorem_subsection}. Instead of $\forall x\in C$ and $\exists x \in C$ 
we now have $\forall x^X$ and $\exists x^X$ (given that the metric space is bounded, this 
quantification runs in parallel with the boundedness of $C$ in the previous section). 
{\em Mutatis mutandis}, we have the characteristic principles $\bAC$ and $\BC$. For the latter, 
the third scheme of bounded collection (see Subsection~\ref{metatheorem_subsection}) takes 
now the form 
$$\forall x^X \exists n^0 A(x,n) \to \exists n \forall x^X \exists k\leq n \, A(x,k),$$ 
where the $A$ is a $\Sigma$-formula (possibly with parameters).  

The theory $\Ms^+$ is the theory $\Ms$ together with the principles $\bAC$ and $\BC$.

\begin{metatheorem_3} Suppose that the theory $\Ms^+$ proves the sentence 
$\forall k^0 \forall f^1 \exists n^0 A(k,\tilde{f},n)$, where $A$ is a $\Sigma$-formula. 
Then there is a closed functional $\phi$ of type $0\to (1 \to 0)$ such that the theory 
$\Ms$ proves $$\forall k \forall f \exists n\leq \phi(k,f) \, A(k,\tilde{f},n).$$ 
\end{metatheorem_3}

The next proposition isolates our Heine/Borel technique (the notation is purposely informal):

\begin{proposition}\label{gp.prop_2} The theory $\Ms^+$ proves the following mathematical 
statement. Let $U$ be a map from $X$ to $X$ and $(u_n)_{n\in\N}$ a sequence of elements of $X$ such 
that $\lim_n d(U(u_n),u_n) = 0$. Consider $F:=Fix(U)$. Given 
$k\in \N$, $\lambda \in \R$ and $\theta: X \to \R$, if 
\begin{equation}\label{gen-principle-hyp-1}
\forall y\in F \left( \lambda < \theta(y) + \frac{1}{k+1}\right)
\end{equation}
then, for $n$ sufficiently large, 
\begin{equation}\label{conclusion-gen-principle}
\lambda < \theta(u_n) + \frac{1}{k+1}
\end{equation}
\end{proposition}
\begin{proof} We reason inside $\Ms^+$. By hypothesis 
$$\forall y\in X\left( \forall m\in \N \left( d(U(y),y)\leq \frac{1}{m+1} \right) \to \lambda < \theta(y) + 
\frac{1}{k+1}\right).$$ 
Hence, $$\forall y \in X  \exists m\in \N \left(d(U(y),y)\leq \frac{1}{m+1} \to \lambda < 
\theta(y) + \frac{1}{k+1}\right).$$ 
By $\BC$, $$\exists l\in \N \forall y\in X \exists m \leq l \left(d(U(y),y)\leq \frac{1}{m+1} 
\to \lambda < \theta(y) + \frac{1}{k+1}\right).$$ Take one such $l = l_0$. 
Clearly $$\forall y \in X \left(d(U(y),y) \leq \frac{1}{l_0+1} \to \lambda < \theta(y) + 
\frac{1}{k+1}\right).$$ Since $\lim_n d(U(u_n),u_n) = 0$, the result follows. \end{proof}
\details{We have that 
\beq
\forall y\in F \left( \lambda < \theta(y) + \frac{1}{k+1}\right)
\eeq
iff 
\beq
\forall y\in X \left( y\in F \to \lambda < \theta(y) + \frac{1}{k+1}\right)
\eeq
iff 
\beq
\forall y\in X \left(U(y)=y \to \lambda < \theta(y) + \frac{1}{k+1}\right)
\eeq
iff 
\beq
\forall y\in X \left(\forall m\in \N \left( d(U(y),y)\leq \frac{1}{m+1} \right) \to \lambda < \theta(y) + \frac{1}{k+1}\right)
\eeq
iff 
\beq\forall y \in X  \exists m\in \N \left(d(U(y),y)\leq \frac{1}{m+1} \to \lambda < 
\theta(y) + \frac{1}{k+1}\right).
\eeq
Let $$A(y,m):= d(U(y),y)\leq \frac{1}{m+1} \to \lambda < 
\theta(y) + \frac{1}{k+1}.$$ Then $A$ is a $\Sigma$-formula, since $\leq_\R$ is $\Pi_1^0$ and $<_\R$ is $\Sigma_1^0$.
Then we can apply (bC) (in the form above) to conclude that 
\beq
\exists l\in \N \forall y\in X \exists m\leq l \,\, A(y,m).
\eeq
Take one such $l = l_0$. Then there exists $m\leq l_0$ such that 
\beq
\forall y\in X \left(d(U(y),y)\leq \frac{1}{m+1} \to \lambda < 
\theta(y) + \frac{1}{k+1}\right).
\eeq
Since $\frac{1}{l_0+1}\leq \frac{1}{m+1}$, we get that 
\beq
\forall y\in X \left(d(U(y),y)\leq \frac{1}{l_0+1} \to \lambda < 
\theta(y) + \frac{1}{k+1}\right).
\eeq
Since $\lim_n d(U(u_n),u_n) = 0$, there exists $N\in\N$ such that, for all $n\geq N$, $d(U(u_n),u_n)\leq \frac{1}{l_0+1}$. 
It follows that for all $n\geq N$,
\beq \lambda < \theta(u_n) + \frac{1}{k+1}.
\eeq
}

In proof mining terms, the crux of the matter of the proposition above reduces to the 
following triviality: $$d(U(u_n),u_n) \leq \frac{1}{r+1} \wedge \forall y\in X \left( d(U(y),y) \leq \frac{1}{r+1} 
\to \lambda < \theta(y) + \frac{1}{k+1}\right) \, \to \,\, \lambda < \theta(u_n) + \frac{1}{k+1}.$$

For the applications in this paper, we need an immediate corollary of the above proposition.

\begin{proposition}[General principle]\label{gp.prop_3} The theory $\Ms^+$ proves the following 
mathematical statement. Let $U$ be a map from $X$ to $X$, $\varphi$ a map from $X \times X$ to 
$\R$, and $(u_n)_{n\in\N}$ be a sequence of elements of $X$ such that $\lim_n d(U(u_n),u_n) = 0$. 
Consider $F:=Fix(U)$. If $$\forall k\in \N \, \exists x \in F \, \forall y\in F 
\, \left(\varphi(x,x) < \varphi(x,y) + \frac{1}{k+1}\right)$$ then $$\forall k\in \N \, 
\exists x \in F \, \exists n\in \N \, \forall m\geq n \left( \varphi(x,x) < \varphi(x,u_m) 
+ \frac{1}{k+1}\right).$$
\end{proposition}

\begin{proof} Given $k \in \N$ take, by hypothesis, $\tilde{x}\in F$ such that 
$\forall y\in F \, (\varphi(\tilde{x},\tilde{x}) < \varphi(\tilde{x},y) + \frac{1}{k+1})$. 
Let $\lambda := \varphi(\tilde{x},\tilde{x})$ and $\theta(y) := \varphi(\tilde{x},y)$. 
Now apply Proposition \ref{gp.prop_2}. 
\end{proof}
We will apply the above proposition to the bounded metric space $C$,
where the metric is induced by the norm. In the case of Browder's theorem, $\varphi(x,y)$ is 
the map $(x,y) \mapsto \langle x-v_0, y\rangle$. 
The next result is a quantitative (mining) version of Proposition \ref{gp.prop_3}. To state 
the result (and other results of this paper), we need the notion of {\em monotone functional}. 
Given $f,g \in \N^\N$, we say that $g \leq^* \!f$ if 
$$\forall n \forall k\leq n \Big(g(k)\leq f(n) \wedge f(k)\leq f(n)\Big).$$ 
Given $\alpha, \beta$ functionals from $\N \times \N^\N$ to $\N$, we say that $\alpha \leq^*\!\beta$ if 
$$\forall n \forall f \forall k\leq n \forall g\leq^*\!\!f \, \Big(\alpha(k,g) \leq \beta(n,f) \wedge \beta(k,g) 
\leq \beta(n,f)\Big).$$ 
This is a particular case of the notion of strong majorizability introduced in \cite{Bezem(85)}. We say that 
$f\in \N^\N$ is monotone if $f\leq^*\!\!f$. Note that this notion coincides with the usual 
notion of monotonicity of a numerical function. We will often need to quantify over monotone functions. 
Instead of writing $\forall f (f \leq^*\!\!f \to \ldots)$, we write $\foralltilde f \, (\ldots)$. Finally, 
a functional $\alpha$ from $\N \times \N^\N$ to $\N$ is monotone if $\alpha \leq^*\! \alpha$. 
(Note that monotone functionals are strongly majorizable in the sense of \cite{Bezem(85)}.)

\begin{proposition}[Quantitative version of the general principle]\label{qv:prop1}
Let $(X,d)$ be a metric space. Let $U$ be a map from $X$ to $X$, $\varphi$ a map from $X \times X$ to $\R$ 
and $(u_n)_{n\in\N}$ be a sequence of elements of $X$. Suppose that there are monotone functionals 
$\alpha$ and $\beta$ from $\N \times \N^\N$ to $\N$ satisfying:
\bi
\item[\rm{(}a\rm{)}]
$\forall k\in \mathbb{N}\, \tilde{\forall} f\in \N^\N\, 
\exists N\leq \alpha(k,f) \, \forall n \in [N,f(N)]\,\left( d(U(u_n),u_n) < \dfrac{1}{k+1}\right)$;
\item[\rm{(}b\rm{)}]
$\forall k\in \mathbb{N}\, \tilde{\forall} f \in \N^\N\, \exists N\leq \beta(k,f)$\\
$\exists x\in X\left( d(U(x),x)< \dfrac{1}{f(N)+1} \,\land\, \forall y \in X\, 
\left( d(U(y),y)\leq \dfrac{1}{N+1} \rightarrow \varphi(x,x) < \varphi(x,y) + \dfrac{1}{k+1}\right)\right)$.
\ei
Then, for every $k \in \N$ and any monotone function $f\in \N^\N$, there is a natural number 
	$N$ with $N\leq \psi(k,f)$ such that
\beq\label{qv:eq1}
\exists x \in X \left( d(U(x),x) < \dfrac{1}{f(N)+1} \, \land \, \forall n\in [N, f(N)] \, 
\left( \varphi(x,x) < \varphi(x,u_n) + \frac{1}{k+1}\right)\right),
\eeq	
where $\psi(k,f)$ is the monotone functional given by $\psi(k,f) := \alpha\left(\beta\left(k,\fseco{f}\right), f\right)$, 
with $\fseco{f}$ the monotone function $m \mapsto f(\alpha(m,f))$.  
\end{proposition} 
\begin{proof} Take $k \in \mathbb{N}$ and a monotone function $f\in \N^\N$. By $(b)$, applied to 
$k$ and $\fseco{f}$ there are $N_1\leq \beta\left(k,\fseco{f}\right)$ and $\tilde{x}\in X$ such that

\begin{equation*}
d(U(\tilde{x}),\tilde{x}) < \dfrac{1}{\fseco{f}(N_1)+1} \, \text{ and}
\end{equation*}

\begin{equation}\label{qv_eq2}
\forall y \in X \left( d(U(y),y) \leq \dfrac{1}{N_1+1} \rightarrow \varphi (\tilde{x}, \tilde{x}) 
< \varphi (\tilde{x},y) + \dfrac{1}{k+1}\right).
\end{equation}

Apply $(a)$ to $N_1$ and $f$ to get $N\leq \alpha(N_1,f)$ satisfying
\begin{equation}\label{qv_eq3}
\forall n \in [N,f(N)] \, \left( d(U(u_n),u_n) < \frac{1}{N_1+1}\right).
\end{equation}

We have $N \leq \alpha(N_1,f) \leq \alpha\left(\beta\left(k,\fseco{f}\right),f\right) = \psi(k,f)$ and, by the monotonicity of $f$,

\begin{equation*}
d(U(\tilde{x}),\tilde{x}) < \dfrac{1}{\fseco{f}(N_1)+1}=\dfrac{1}{f\left(\alpha(N_1,f)\right)+1}\leq \dfrac{1}{f(N)+1}.
\end{equation*}

Also, for $n \in [N,f(N)]$, by \eqref{qv_eq2} and \eqref{qv_eq3}, we have
\begin{equation*}
\varphi(\tilde{x},\tilde{x}) < \varphi (\tilde{x},u_n) + \dfrac{1}{k+1}.
\end{equation*}
\end{proof}
There are several differences between Propositions~\ref{gp.prop_2} and \ref{gp.prop_3}. However, \
the difference worth remarking is that $x$ is a fixed point in the latter proposition. 
This is essential for obtaining the mining of Browder's theorem. The statement with 
an arbitrary $x \in X$ would simplify a bit the bounds of Proposition~\ref{qv:prop1}.

\subsection{The general principle is false}\label{false_principle}

The mathematical statement of Proposition~\ref{gp.prop_2} is false, even when $X$ is a (bounded) 
complete metric space, $U$ is continuous and has fixed points, and $\theta$ is bounded and continuous. 
For a counterexample, take $X$ as the unit ball of the normed space $\ell^1$ (the space 
of real-valued sequences whose series is absolutely convergent). Let $U$ be the shift operation 
$U(x_0,x_1,x_2,\ldots) := (0,x_0,x_1, \ldots)$ and $\theta(x_0,x_1,x_2,\ldots) = - 
\sum_{i=0}^\infty |x_i|$ (i.e., $\theta(x)$ is the symmetric of the norm of $x$). Note that 
the only fixed point of $U$ is the zero vector. Let $u_n$ be the vector 
$(\frac{1}{n+1},\ldots,\frac{1}{n+1},0,0,\ldots)$, where there are $n+1$ nonzero entries. 
Clearly, $\theta(u_n) = -1$ and $\|U(u_n) - u_n\| = \frac{2}{n+1}$. The mathematical statement 
of 
Proposition~\ref{gp.prop_2} fails in this case for $\lambda := 0$ (for any natural number $k$).

The general principle also fails in this setting. Just consider $\varphi(x,y) := -\|y\|$.

\details{Using the notations of Proposition \ref{gp.prop_2}, we have that $\lim_n \|U(u_n) - u_n\|=0$. Furthermore,
\[\forall y\in F \left( \lambda < \theta(y) + \frac{1}{k+1}\right) \Lra 0<\theta(0)+ \frac{1}{k+1}\Lra 0 < 0+\frac{1}{k+1}\]
is true, while
\[\lambda < \theta(u_n) + \frac{1}{k+1} \Lra 0<-1+\frac{1}{k+1}\Lra 1<\frac{1}{k+1}\]
is obviously false. Thus, Proposition \ref{gp.prop_2} is false. 

Let us look at the general principle. We have that 
$$\forall k\in \N \, \exists x \in F \, \forall y\in F 
\, \left(\varphi(x,x) < \varphi(x,y) + \frac{1}{k+1}\right)$$ 
iff
$$ \forall k\in \N \, 
\, \left(\varphi(0,0) < \varphi(0,0) + \frac{1}{k+1}\right)$$
is true, while
$$\forall k\in \N \, 
\exists x \in F \, \exists n\in \N \, \forall i\geq n \left( \varphi(x,x) < \varphi(x,u_i) 
+ \frac{1}{k+1}\right)$$
iff 
$$\forall k\in \N \, \exists n\in \N \, \forall i\geq n \left( \varphi(0,0) < \varphi(0,u_i) 
+ \frac{1}{k+1}\right)$$
iff 
$$\forall k\in \N \, \exists n\in \N \, \forall i\geq n \left( 0 < -\|u_i| 
+ \frac{1}{k+1}\right)$$
iff 
$$\forall k\in \N \, \exists n\in \N \, \forall i\geq n \left( 0 < -1
+ \frac{1}{k+1}\right),$$
which is obviously false.
}

\subsection{The explanation of the bounded functional interpretation}\label{bfi_explanation}

The reader may have noticed that we did not require in Proposition~\ref{qv:prop1} the boundedness 
of the metric space $(X,d)$. As can be seen from the proof, this hypothesis is not necessary. 
Proposition \ref{qv:prop1} is a just a simple mathematical fact. In our applications, however, 
the given monotone functionals $\alpha$ and $\beta$ depend on the bound of the metric space 
(as well as the concluding bounding functional $\psi$). A similar situation also happens 
in the forthcoming Proposition \ref{qv:prop2}.

The full bounded functional interpretation can explain the form (given in Proposition~\ref{qv:prop1}) 
taken by the quantitative version of Proposition \ref{gp.prop_3} (provided that the space is bounded). 
Let us try to explain this (in the explanation, we must use some facts that can be found in 
\cite{FerreiraOliva(05)} and \cite{Engracia(09)}). In order to simplify the exposition, we 
assume that the function $\varphi$ is bounded (this is actually automatic in the bounded 
functional interpretation; naturally, the bounds obtained by the mining may depend on the 
bound of $\varphi$). 

The mathematical statement of Proposition \ref{gp.prop_3} is of the following sort: 
$$\forall U^{X\to X} \forall \varphi^{X\to (X \to 1)} \forall u^{0\to X} \, (\textrm{Hyp}_1 
\wedge \textrm{Hyp}_2 \to \text{\rm Con}),$$ 
where (Hyp$_1$) is $\lim_n d(U(u_n),u_n) = 0$, 
(Hyp$_2$) is the other assumption and (Con) is the conclusion. Using $\BC$ and $\bAC$ and 
ignoring the parameters $U$, $\varphi$ and $u$ for the moment, we can put the three formulas 
(Hyp$_1$), (Hyp$_2$) and (Con) in quantitative form. Therefore, the mathematical statement of Proposition~\ref{qv:prop1} 
takes the form $$\forall r,g\exists m \, A(r,\tilde{g},m) \wedge \forall k,f \exists n \, 
B(k,\tilde{f},n) \to \forall s,h \exists q \, C(s,\tilde{h},q),$$ where $C$ is a $\Sigma$-formula, 
and $A$ and $B$ are (for technical reasons) bounded formulas. In its fullest generality, 
the bounded choice principle applies to all finite types. In this case, we get 
$$\tilde{\exists} \alpha \forall r,g \exists m\leq \alpha(r,\tilde{g}) A(r,\tilde{g},m) \wedge 
\tilde{\exists} \beta \forall k,f \exists n \leq \beta(k,\tilde{f}) B(k,\tilde{f},n) \to 
\forall s,h \exists q \, C(s,\tilde{h},q).$$ 
Here $\alpha$ and $\beta$ are of type $0 \to (1 \to 0)$ and they are monotone in the intensional 
sense (this is what the tildes above the quantifiers mean). 
The reader can, however, ignore these (and similar) fine points and still get the gist of the explanation. 
Therefore, we have 
$$\forall s, h \, \tilde{\forall} \alpha, \beta \, \exists q \, 
\big[\forall r,g \exists m\leq \alpha(r,\tilde{g}) A(r,\tilde{g},m) \wedge \forall k,f 
\exists n \leq \beta(k,\tilde{f}) B(k,\tilde{f},n) \to C(s,\tilde{h},q)\big].$$ 

What about the parameters or, better still, the block of universal quantifiers 
$\forall U^{X\to X} \forall \varphi^{X\to (X \to 1)} \forall u^{0\to X}$? Since $X$ is bounded 
(as well as the map $\varphi$), all these quantifications are classified as bounded quantifications. 
So, displaying all the variables, we have 
$$\forall^{\rm b} U \forall^{\rm b} \varphi \forall^{\rm b} u \, 
\forall s, h\, \tilde{\forall} \alpha, \beta \, \exists q \, \big[\forall r,g \exists m\leq \alpha(r,\tilde{g}) 
A(r,\tilde{g},m,U,\varphi,u) \wedge \forall k,f \exists n \leq \beta(k,\tilde{f}) B(k,\tilde{f},n,U,\varphi,u) $$ 
$$\to C(s,\tilde{h},q,U,\varphi,u)\big],$$ 
where the quantifiers $\forall^b$ denote bounded quantifications. The formula between square parentheses is a $\Sigma$-formula in an appropriate sense (this is the reason 
why we required that the formulas $A$ and $B$ be bounded). In its fullest generality, 
the bounded collection principle applies to all finite types. In our case, we get 
$$\forall s, h \tilde{\forall} \alpha, \beta \, \exists l \Big[\,\forall^{\rm b} U 
\forall^{\rm b} \varphi \forall^{\rm b} u \,\exists q \leq l  \big(\forall r,g 
\exists m\leq \alpha(r,\tilde{g}) A(r,\tilde{g},m,U,\varphi,u) \wedge \forall k,f 
\exists n \leq \beta(k,\tilde{f}) B(k,\tilde{f},n,U,\varphi,u) $$ $$\to C(s,\tilde{h},q,U,\varphi,u)\big)\Big].$$ 
Once again, the formula between square parentheses is a $\Sigma$-formula in an appropriate sense.  
The full metatheorem guarantees the existence of a closed term $t$ in G\"odel's {\sf T} such 
that 
$$\forall s, h \tilde{\forall} \alpha, \beta \, \exists l \leq t(s,h,\alpha,\beta) \,
\forall^{\rm b} U \forall^{\rm b} \varphi \forall^{\rm b} u \,\exists q \leq l  
\big[\big(\forall r,g \exists m\leq \alpha(r,\tilde{g}) A(r,\tilde{g},m,U,\varphi,u) \wedge$$ 
$$\forall k,f \exists n \leq \beta(k,\tilde{f}) B(k,\tilde{f},n,U,\varphi,u)\big) \to C(s,\tilde{h},q,U,\varphi,u)\big].$$ 
Therefore, for any $U$, $\varphi$ and $u$, one has 
$$\forall s, h \tilde{\forall} \alpha, \beta \exists q \leq t(s,h,\alpha,\beta) 
\big[\forall r,g \exists m\leq \alpha(r,\tilde{g}) A(r,\tilde{g},m,U,\varphi,u) \wedge 
\forall k,f \exists n \leq \beta(k,\tilde{f}) B(k,\tilde{f},n,U,\varphi,u)$$ 
$$\to C(s,\tilde{h},q,U,\varphi,u)\big].$$ 
Notice that the bound given by $t$ does not depend on 
$U$, $\varphi$ or $u$ (uniformity of the bound). Of course, this is what happens in Proposition~\ref{qv:prop1}, 
where the bound $\psi$ only depends on $k,f$, $\alpha$ and $\beta$.

The form of the quantitative version is almost fully explained. For the full explanation, one must first notice that (a) of Proposition~\ref{qv:prop1} comes from the metastable  version of (Hyp$_1$). The reader may worry that the matrix of this formula is a  $\Sigma$-formula, not a bounded formula (as was required, for technical reasons).  What happens is that in the proper setting of the bounded functional interpretation the matrix would indeed be a bounded formula, were we allowed to use intensional majorizability relations. A similar situation also occurs in the forthcoming analysis of (Hyp$_2$). 

Secondly, the conclusion of Proposition~\ref{qv:prop1} comes from working out a quantitative form of 
(Con) as in the case of the projection statement (\ref{near_closest}), done at the beginning 
of Section~\ref{projection}. 

Finally, it remains to work out a quantitative form of (Hyp$_2$). Well, (Hyp$_2$) is 
$$\forall k \exists n \exists x\in X \left(\forall r \left( d(U(x),x) \leq 
\frac{1}{r+1}\right) \wedge \forall i\geq n \left(\varphi(x,x) \leq \varphi(x,u_i) + 
\frac{1}{k+1}\right) \right).$$ This is equivalent to $$\forall k \exists n \exists x\in X \forall r 
\left(d(U(x),x) \leq \frac{1}{r+1} \wedge \left(r\geq n \to \varphi(x,x) \leq \varphi(x,u_r) + 
\frac{1}{k+1}\right)\right). $$ 

Using (the contrapositive of) $\BC$, we get 
$$\forall k \exists n \forall r \exists x\in X \left(d(U(x),x) \leq \frac{1}{r+1} \wedge \forall i\in [n,r] 
\left(\varphi(x,x) \leq \varphi(x,u_i) + \frac{1}{k+1}\right) \right).$$ 
By (the contrapositive of) $\bAC$, 
we get $$\forall k \forall f \exists n \, \exists x\in X\left( d(U(x),x) \leq \frac{1}{\tilde{f}(n)+1} \wedge 
\forall i \in [n,\tilde{f}(n)] \left( \varphi(x,x) \leq \varphi(x,u_i) + \frac{1}{k+1}\right)\right).$$ 
By fiddling with $k$ and $f$, we get $$\forall k \forall f \exists n \left[\exists x\in X\left( d(U(x),x) 
< \frac{1}{\tilde{f}(n)+1} \wedge \forall i \in [n,\tilde{f}(n)] \left( \varphi(x,x) < \varphi(x,u_i) + 
\frac{1}{k+1}\right)\right)\right].$$ 
Note the change to strict inequalities. 
This move turns the formula between square parentheses into a $\Sigma$-formula (up to equivalence). 
This explain the form that (b) takes in Proposition~\ref{gp.prop_3}.

\section{The minings of the theorems of Browder and Wittmann}\label{browder-wittmann}

In this section we shall obtain, using the quantitative general principle, uniform effective versions of the first two theorems: Browder's and Wittman's. By an analysis of the proofs of these results, one can see that they finish with a simple argument that relies on an application of {\em modus ponens} (and the triangle inequality). Let us isolate this argument.

\begin{lemma}\label{modusponens} 
The theory $\Ms$ proves the following mathematical statement. 
Let $U$ be a map from $X$ to $X$, $\varphi$ a map from $X \times X$ to $\R$ and $(u_n)_{n\in\N}$ be a 
sequence of elements of $X$. Consider $F:=Fix(U)$. 
Suppose that 
$$\forall k \in \N \, \exists x \in F \, \exists n \in \N \, \forall m\geq n 
\left( \varphi(x,x) \leq \varphi(x,u_m) + \frac{1}{k+1} \right)$$ 
and that there is a monotone function $\delta: \N \to \N$ such that, for all $k\in \N$ and $x\in F$, 
$$\exists n \in \N \, \forall m\geq n \left( \varphi(x,x) \leq \varphi(x,u_m) + \frac{1}{\delta(k)+1} \right) \, 
\to \, \exists M \in \N \,\forall m\geq M \left(d(u_m,x) \leq \frac{1}{k+1}\right).$$ 
Then, $(u_n)$ is a Cauchy sequence. 
\end{lemma}

Clearly, the lemma is correct. Following the explanation of Subsection~\ref{bfi_explanation}, it is not 
difficult to find the form that the quantitative version of the 
above lemma must take under the 
bounded functional interpretation. A quantitative form of the first assumption was worked out at 
the end of Subsection~\ref{bfi_explanation}.  

The second assumption is $$\forall k \forall x\in X\left[\forall r \left(d(U(x),x) \leq \frac{1}{r+1} \right) 
\wedge \exists n \forall i\geq n \left(\varphi(x,x) \leq \varphi(x,u_i) + \frac{1}{\delta(k) + 1}\right) \right.$$ 
$$\left. \to \,  \exists M \forall m\geq M \left(d(u_m, x) \leq \frac{1}{k+1} \right) \right].$$ 
By (the contrapositive of) $\bAC$, the conclusion of the implication
is equivalent to 
$$\forall f \exists M \forall m\in [M,\tilde{f}(M)] \left(d(u_m,x) \leq \frac{1}{k+1} \right).$$ 
Therefore, the second assumption is equivalent to 
$$\forall k,n \forall f \forall x\in X \exists r \exists i\geq n \exists M \left[ d(U(x),x) \leq \frac{1}{r+1} 
\wedge \varphi(x,x) \leq \varphi(x,u_i) + \frac{1}{\delta(k) + 1} \right.$$ 
$$\left. \to \, \forall m \in [M,\tilde{f}(M)] \left(d(u_m, x) < \frac{1}{k+1} \right) \right].$$ 
Notice the change to strict inequality at the end. Since the formula in square brackets is equivalent 
to a $\Sigma$-formula, by $\BC$ we have 
$$\forall k,n \forall f  \exists r, i, M \forall x\in X \exists \check{r}\leq r \exists j\in [n, i] 
\exists \check{M}\leq M \left[ d(U(x),x) \leq \frac{1}{\check{r}+1} \wedge \varphi(x,x) \leq \varphi(x,u_j)
 + \frac{1}{\delta(k) + 1} \right.$$ 
 $$\left. \to\,\forall m \in [\check{M},\tilde{f}(\check{M})] \left(d(u_m, x) < \frac{1}{k+1} \right) \right].$$ 
 Therefore $$\forall k,n \forall f  \exists r, i, M \forall x\in X \left[ d(U(x),x) \leq \frac{1}{r+1} 
 \wedge \forall j \in [n,i]\left(\varphi(x,x) \leq \varphi(x,u_j) + \frac{1}{\delta(k) + 1}\right) \right.$$ 
 $$\left. \to \exists \check{M}\leq M \forall m \in [\check{M},\tilde{f}(\check{M})] \left(d(u_m, x) < \frac{1}{k+1} \right)
  \right].$$

The quantitative version of Lemma~\ref{modusponens} takes, then, the following form:

\begin{proposition}[Quantitative version of \ref{modusponens}]\label{qv:prop2}
Let $(X,d)$ be a metric space. Let $U$ be a map from $X$ to $X$, $\varphi$ a map from 
$X \times X$ to $\R$ and $(u_n)_{n\in\N}$ be a sequence of elements of $X$. Suppose that there are 
monotone functions $\delta, \psi, \gamma, \eta$ and $\sigma$ satisfying:
\begin{enumerate}
\item[\rm{(}i\rm{)}] $\forall k \in \mathbb{N} \,\tilde{\forall} f \in \N^\N\, \exists N\leq \psi(k,f)$\\
$\exists x\in X \left( d(U(x),x) < \dfrac{1}{f(N)+1} \, \land \, \forall n\in [N, f(N)] \, 
\left( \varphi (x,x) < \varphi (x,u_n) + \dfrac{1}{k+1}\right)\right)$\, \text{ and}
\item[\rm{(}ii\rm{)}] $\forall k, n \in \N \, \tilde{\forall} f\in \N^\N \, \forall x \in X $\\ 
$\Big[ d(U(x),x)\leq \dfrac{1}{\gamma(k,n,f)+1} \,\land\, \forall i\in [n,\eta(k,n,f)]\, 
\Big( \varphi (x,x) \leq \varphi (x,u_i) + \dfrac{1}{\delta(k)+1}\Big)$\\
$~\qquad\qquad\rightarrow \exists M\leq \sigma(k,n,f)\, \forall m\in[M,f(M)]\, 
\Big(d(u_m,x) < \dfrac{1}{k+1}\Big)\Big]$.
\end{enumerate}
Then
\begin{equation}
\forall k \in \mathbb{N} \,\tilde{\forall} f \in \N^\N\, \exists M\leq \phi(k,f)\, 
\forall m,n\in[M,f(M)]\, \Big( d(u_m,u_n) < \dfrac{1}{k+1}\Big),
\end{equation}
where $\phi(k,f):=\sigma\Big(2k+1, \psi\big(\delta(2k+1),\fsec{f}\big), f\Big)$ and 
$\fsec{f}(m):=\max\{ \gamma(2k+1,m,f),\, \eta(2k+1,m,f)\}$.
\end{proposition}
\begin{proof}
Let $k\in \mathbb{N}$ and monotone $f\in \N^\N$ be given. We apply condition $(i)$ to 
$\delta(2k+1)$ and $\fsec{f}$ in order to get $N_1\leq \psi(\delta(2k+1),\fsec{f})$ and $\tilde{x}\in X$ such that
\begin{align}
&d(U(\tilde{x}),\tilde{x}) < \dfrac{1}{\fsec{f}(N_1)+1}\, \text{ and}\nonumber\\
&\forall n\in[N_1,\fsec{f}(N_1)]\, \Big( \varphi (\tilde{x},\tilde{x}) < \varphi (\tilde{x},u_n) 
+ \dfrac{1}{\delta(2k+1)+1}\Big).\label{qv:eq5}
\end{align}

Now apply $(ii)$ to $2k+1, N_1, f$ and $\tilde{x}\in X$ and obtain 
\begin{align} d(U(\tilde{x}),\tilde{x})\leq \dfrac{1}{\gamma(2k+1,N_1,f)+1}\, \land\, 
\forall i \in [N_1,\eta(2k+1,N_1,f)]\, \Big( \varphi (\tilde{x},\tilde{x}) \leq 
\varphi (\tilde{x},u_i) + \dfrac{1}{\delta(2k+1)+1}\Big) \nonumber \\ 
\rightarrow \exists M\leq \sigma(2k+1,N_1,f)\,\forall m\in[M,f(M)]\, \Big( d(u_m,\tilde{x})< \dfrac{1}{2k+2}\Big).
\label{qv:eq6}
\end{align}

Since $\gamma(2k+1,N_1,f),\, \eta(2k+1,N_1,f)\leq \fsec{f}(N_1)$, by \eqref{qv:eq5} we have 
the antecedent of \eqref{qv:eq6}. Therefore $$\exists M\leq \sigma(2k+1,N_1,f)\,\forall m\in[M,f(M)]\, 
\left( d(u_m,\tilde{x}) < \dfrac{1}{2k+2}\right)$$

Finally, we have $M\leq \sigma(2k+1,N_1,f)\leq 
\sigma\big(2k+1, \psi\big(\delta(2k+1),\fsec{f}\big),f\big)=\phi(k,f)$ and, 
by the triangle inequality, the result follows. 
\end{proof}

\subsection{The mining of Browder's theorem}

In the following, we are in the hypotheses of Browder's Theorem~\ref{browder}. Thus, $X$ is a real Hilbert space, $C$ is a bounded closed convex subset of $X$, $U:X\to X$ is a nonexpansive mapping that maps $C$ into itself, $v_0\in C$ and the sequence $(u_n)$ is defined as in Theorem~\ref{browder}.

We use the quantitative general principle (Proposition~\ref{qv:prop1}) and Proposition~\ref{qv:prop2} for the bounded metric space $C$ with the metric induced by the Hilbert space norm and for the mapping $\varphi(x,y):=\langle x - v_0,y\rangle$. Let $b\in\N^*$ be an upper bound on the diameter of $C$. Let us define first the following functions:
\beq
\usr:\N\to\N, \quad  \usr(k)=b^4(k+1)^2+b^2.\label{def-usr} 
\eeq 
and, for every $g:\N\to\N$, 
\beq
\usf_g: \N\to \N, \quad  \usf_g(m)= \max\{g(12b(m+1)^2),\, 12b(m+1)^2 \}+1.\label{def-usf}
\eeq

As an immediate consequence of \eqref{quasi_fixed} of Section~\ref{modified}, we get that condition (a) of Proposition~\ref{qv:prop1} is fulfilled with 
$$N:=\alpha(k,f) := b(k+1).$$

\details{By (\ref{quasi_fixed}), we have that for all $n$,\[ \Vert U(u_n) - u_n \Vert \leq \frac{b}{n+1}.\] Let $k\in\N$ and $f\in \N^\N$ monotone. Then for all $n\geq N:=\alpha(k,f)$, we have that \[ d(U(u_n),u_n))\leq \dfrac{b}{n+1}\leq \dfrac{b}{\alpha(k,f)+1}\leq \dfrac{b}{b(k+1)}=\dfrac{1}{k+1}.\] In particular, (a) from Proposition \ref{qv:prop1} holds.}

Furthermore, condition (b) of Proposition~\ref{qv:prop1} is satisfied with $$\beta(k,f):=12b\left(\usf_f^{(\usr(k))}(0)+1\right)^2.$$This was worked out in Proposition~\ref{proj3}.

\details{Clear.}

Therefore we can apply Proposition~\ref{qv:prop1} in order to obtain condition $(i)$ of Proposition~\ref{qv:prop2} with $$\psi(k,f):=12b^2\left(\usf_{\fseco{f}}^{(\usr(k))}(0)+1\right)^2+b,$$ where  $\fseco{f}(m):=f(\alpha(m,f))=f(b(m+1))$. 
\details{
Since the hypotheses of Proposition \ref{qv:prop1} are satisfied, we can apply it to get that 
for every $k \in \N$ and any monotone function $f\in \N^\N$, there exists $N\leq \psi(k,f)$ such that
\begin{equation}\label{qv:eq1}
\exists x \in X \left( d(U(x),x) < \dfrac{1}{f(N)+1} \, \land \, \forall n\in [N, f(N)] \, 
\left( \varphi(x,x) < \varphi(x,u_n) + \frac{1}{k+1}\right)\right),
\end{equation}	
where 
\[\fseco{f}(m)=f(\alpha(m,f))=f(b(m+1)-1)\]
and
\bua \psi(k,f) &=& \alpha\left(\beta(k,\fseco{f}), f\right) = b\left(\beta\left(k,\fseco{f}\right)+1\right)-1=
b\left(12b\left(\usf_{\fseco{f}}^{(\usr(k))}(0)+1\right)^2+1\right)-1\\
&=&
12b^2\left(\usf_{\fseco{f}}^{(\usr(k))}(0)+1\right)^2+b-1.
\eua
}
We now need to show that  hypothesis $(ii)$ of Proposition~\ref{qv:prop2} holds. This follows from the mining of (II) of Section~\ref{modified}. It can be read from Kohlenbach's computations of Lemma 2.11 of  \cite{Kohlenbach(11)} that, for all $x\in C$ and $k,n\in \N$, 

\begin{equation}\label{qlem-Koh}
\|U(x) - x \| \leq \frac{1}{2b(n+1)(k+1)^2+1} \wedge \langle x-v_0,x-u_n\rangle \leq \frac{1}{2(k+1)^2} \,\, \to \,\, \|u_n-x\| <\frac{1}{k+1}.
\end{equation}

Therefore, condition $(ii)$ of Proposition \ref{qv:prop2} holds with 
\bce
$\gamma(k,n,f):=2b(f(n)+1)(k+1)^2$, $\delta(k):=2(k+1)^2-1$, $\eta(k,n,f):=f(n)$ and $M:=\sigma(k,n,f):=n$.
\ece

\details{
Let $k,n\in\N$, $f\in \N^\N$ be monotone and $x\in X$ be such that the premise of (ii) holds with 
$\gamma,\delta, \eta,\sigma$ defined as above. Let $M:=\sigma(k,n,f)=n$. We have to prove that 
for all $m\in[n,f(n)]$, $d(u_m,x) < \dfrac{1}{k+1}$. 

Let $m\in[n,f(n)]$ be arbitrary. Then, we get, by the premises of (ii) that 
\bua 
d(U(x),x)\leq \dfrac{1}{\gamma(k,n,f)+1} &\Lra & \|U(x) - x \| \leq \frac{1}{2b(f(n)+1)(k+1)^2} \\
& \Ra & \|U(x) - x \| \leq \frac{1}{2b(m+1)(k+1)^2} \quad \text{since~} m\leq f(n)\\
\langle x-v_0,x-u_n\rangle &=& \varphi (x,x) - \varphi (x,u_i)\leq \dfrac{1}{\delta(k)+1}=\frac{1}{2(k+1)^2}
\eua
Thus, we can apply \eqref{qlem-Koh} with $x,k,m$ to get that $d(u_m,x) <\dfrac{1}{k+1}$.
}

Finally, the conclusion of Proposition~\ref{qv:prop2} yields:

\begin{theorem}[Quantitative Browder]\label{app:quant-browder}
Under the conditions of Browder's theorem, let $b\in\N^*$ be an upper bound on the diameter of $C$. Then, for all $k\in \N$ and every monotone function $f:\N \to \N$, $$\exists N\leq \phi_b(k,f) \,\forall i,j \in [N,f(N)] \, \left( \|u_i-u_j\|< \dfrac{1}{k+1}\right),$$ where $$\phi_b(k,f):=12b^2 \left(h^{(R)}(0)+1\right)^2+b,$$ with $R:=64b^4(k+1)^4+b^2$ and $h(m):= \max\{ 8b(f(12b^2(m+1)^2 + b)+1)(k+1)^2,12b(m+1)^2\}+1.$
\end{theorem}

\begin{proof}
Apply Proposition~\ref{qv:prop2} and remark that
\bua
\fsec{f}(m) &=& 8b(f(m)+1)(k+1)^2,\\
\phi_b(k,f) & = & \sigma\left(2k+1, \psi\left(\delta(2k+1),\fsec{f}\right), f\right)=\psi\left(\delta(2k+1),\fsec{f}\right)
= 12b^2\left(\usf_{\fseco{\fsec{f}}}^{(\usr(\delta(2k+1)))}(0)+1\right)^2+b\\
&=& 12b^2\left(h^{(R)}(0)+1\right)^2+b.
\eua
\end{proof}

\details{Since (i) and (ii) from  Proposition \ref{qv:prop2} are satisfied, we can apply it 
to get that for all $k\in \N$ and every monotone function $f:\N \to \N$,
$$\exists N\leq \phi_b(k,f) \,\forall i,j \in [N,f(N)] \, \left( \|u_i-u_j\|< \dfrac{1}{k+1}\right),$$ 
where 
\bua
\phi_b(k,f) & := & \sigma\left(2k+1, \psi\left(\delta(2k+1),\fsec{f}\right), f\right)=
\psi\left(\delta(2k+1),\fsec{f}\right)\\
&=& 12b^2\left(\usf_{\fseco{\fsec{f}}}^{(\usr(\delta(2k+1)))}(0)+1\right)^2+b\\
&=& 12b^2\left(h^{(R)}(0)+1\right)^2+b
\eua
since
\bua 
\fsec{f}(m) &=& \max\{ \gamma(2k+1,m,f),\, \eta(2k+1,m,f)\}\\
&=& \max\{2b(f(m)+1)(2k+2)^2-1, f(m)\}=\max\{8b(f(m)+1)(k+1)^2-1, f(m)\}\\
&=& 8b(f(m)+1)(k+1)^2-1\\
\fseco{\fsec{f}}(m) &=& \fsec{f}(b(m+1)-1)=8b(f(b(m+1)-1)+1)(k+1)^2-1\\
\usf_{\fseco{\fsec{f}}}(m)&=&\max\{\fseco{\fsec{f}}(12b(m+1)^2), \, 12b(m+1)^2\}\\
&=& \max\{8b\bigg(f\big(b(12b(m+1)^2+1)-1\big)+1\bigg)(k+1)^2-1,\, 12b(m+1)^2\}\\
&=& \max\{8b(f(12b^2(m+1)^2+b-1)+1)(k+1)^2-1,\, 12b(m+1)^2\}\\
&=& h(m)\\
\usr(\delta(2k+1)) &=& b^4(\delta(2k+1)+1)^2=b^4(2(2k+2)^2)^2=b^4\big(8(k+1)^2\big)^2\\
&=& 64b^4(k+1)^4 \\
&=& R
\eua
}

\subsection{The mining of Wittmann's theorem}\label{app_wittmann}

In this section we show how, using Propositions~\ref{qv:prop1} and \ref{qv:prop2}, 
we can mine the proof of a special case of the following important result proved by 
Rainer Wittmann in \cite{Wittmann(92)}.

\begin{theorem}[Wittmann]\label{app:wittmann}
Let $X$ be a Hilbert space, $C$ be a nonempty closed convex bounded subset of 
$X$ and  $U:C\to C$ be a nonexpansive mapping. 

Assume that $(\lambda_n)_{n\in \N}$ is a sequence in $(0,1)$  satisfying
\[\ba{lll} 
(\rm{C}1) \quad  \lim \lambda_n=0, & (\rm{C}2) \quad \ds\sum_{n=1}^\infty\lambda_n=\infty, & (\rm{C}3)  
\quad \ds\sum_{n=1}^\infty|\lambda_n-\lambda_{n+1}|<\infty.
\ea\]
Let $u\in C$ and define the sequence $(u_n)_{n\in\N}$ as follows:
\beq\label{def-Halpern}
u_0:=u, \quad u_{n+1}:=\lambda_{n+1}u_0+ (1-\lambda_{n+1})U(u_n). 
\eeq
Then $(u_n)$ converges strongly to a fixed point of $U$ in $C$  {\em (}the closest one to $u${\em )}. 
\end{theorem}

The sequence $(u_n)$ is known as the {\em Halpern iteration}, studied for the first time by Benjamin Halpern in \cite{Halpern(67)} in the particular case $u=0$. One can easily see that $\lambda_n:=\frac1{n+1}$  satisfies conditions $(\rm{C}1)$-$(\rm{C}3)$. Furthermore, if $U$ is linear and $\lambda_n:=\frac1{n+1}$, then  the Halpern iteration becomes the usual ergodic average. Thus,  Wittmann's result is a nonlinear  generalization of the von Neumann mean ergodic theorem and $(u_n)$ can be seen as a nonlinear ergodic average.

In the sequel, for simplicity, we obtain a quantitative version of Wittmann's theorem in the case 
$\lambda_n:=\frac1{n+1}$.

As in the case of Browder's theorem, we work with the bounded metric space $C$ with the metric induced by the Hilbert space norm. Let $b\in\N^*$ be an upper bound on the diameter of $C$. This time, we use the mapping $$\varphi(x,y):=\langle x - u_0,U(y)\rangle.$$ We must adjust Proposition \ref{proj3} to the new $\varphi$. For every $k\in\N$ and every $g:\N\to\N$, let
\beq
\usw_{k,g}:\N\to\N \quad \mbox{with} \quad \usw_{k,g}(m)=\max\{g(m),2b(k+1)\}.\label{def-usw}
\eeq

\begin{proposition}\label{proj4}
For any $k\in \N$ and monotone $f:\N \to \N $, there is $N \in \N$ such that 
$N\leq 12b(\usf_{\usw_{k,f}}^{(\usr(2k+1))}(0)+1)^2$ and 
$$\exists x\in C\, \left( \|U(x)-x\| < \frac{1}{f(N)+1} \wedge \, \forall y\in C \left(\|U(y)-y\|\leq \frac{1}{N+1} 
\to \langle x - u_0, U(x)-U(y)\rangle < \frac{1}{k+1}\right)\right),$$ 
where $\usf_{\usw_{k,f}}$ is defined by \eqref{def-usf} and $r$ is 
defined by \eqref{def-usr}.
\end{proposition}
\begin{proof} Let $k\in\N$ and $f:\N\to\N$ be monotone. Applying Proposition \ref{proj3} to 
$2k+1$ and to the monotone function $\usw_{k,f}$, we get 
the existence of  $x\in C$  and $N \in \N$, $N\leq 12b(\usf_{\usw_{k,f}}^{(\usr(2k+1))}(0)+1)^2$ such 
that
\begin{equation}
\|U(x)-x\| < \frac{1}{\usw_{k,f}(N)+1}  \, \wedge \, \forall y\in C\, \left(\|U(y)-y\|\leq \frac{1}{N+1} \to 
\langle x-u_0,x-y\rangle < \frac{1}{2(k+1)}\right). 
\label{projW.eq3}
\end{equation}
Since $\usw_{k,f}(N)\geq f(N)$, we have  that 
\beq\label{proj4-l1}
\|U(x) - x\| < \dfrac{1}{f(N)+1}.
\eeq
Let now $y\in C$ be such that $\|U(y)-y\|\leq \frac{1}{N+1}$. As $U$ is nonexpansive, we 
also have that $\|U(U(y)) - U(y)\| \leq \frac{1}{N+1}$. Thus, we can apply \eqref{projW.eq3} to
conclude that 
\begin{equation*}
\langle x-u_0, x-U(y)\rangle < \frac{1}{2(k+1)}.
\end{equation*}
Since 
\bua
\langle x-u_0,U(x)-U(y)\rangle &\leq & \langle x-u_0, U(x)-x\rangle + \langle x-u_0,x-U(y)\rangle \\
&\leq & b\| U(x)-x\| + \frac{1}{2(k+1)} < \frac{b}{\usw_{k,f}(N)+1} + \frac{1}{2(k+1)}
\eua
and $\usw_{k,f}(N)\geq 2b(k+1)$, it follows that 
\beq\label{proj4-l2}
\langle x-u_0,U(x)-U(y)\rangle < \frac{1}{k+1}.
\eeq
By \eqref{proj4-l1} and \eqref{proj4-l2}, the result follows.
 \end{proof}

Hence,  condition (b) of Proposition \ref{qv:prop1} holds with 
\begin{equation*}
\beta(k,f):=12b\left(\usf_{\usw_{k,f}}^{(\usr(2k+1))}(0)+1\right)^2.
\end{equation*}

The bound $\alpha$ for (a) of Proposition~\ref{qv:prop1} was computed in \cite[Lemma 3.1]{Kohlenbach(11)}:
\begin{equation*}
\alpha(k,f):=4b(k+1)\big(4b(k+1)+2\big)=16b^2(k+1)^2+8b(k+1).
\end{equation*}

From the conclusion of Proposition \ref{qv:prop1}, we get that the condition $(i)$ of Proposition \ref{qv:prop2}
holds with 
\bua
\psi(k,f) &=& \alpha\left(\beta\left(k,\fseco{f}\right), f\right) = 16b^2\left(\beta\left(k,\fseco{f}\right)+1\right)^2+
8b\left(\beta\left(k,\fseco{f}\right)+1\right)\\
&=& 16b^2\left(12b\left(\usf_{\usw_{k,\fseco{f}}}^{(\usr(2k+1))}(0)+1\right)^2+1\right)^2+
8b\left(12b\left(\usf_{\usw_{k,\fseco{f}}}^{(\usr(2k+1))}(0)+1\right)^2+1\right),
\eua
where $\fseco{f}(m):=f(\alpha(m,f))=f(16b^2(m+1)^2+8b(m+1))$.

\details{
Since the hypotheses of Proposition \ref{qv:prop1} are satisfied, we can apply it to get that 
for every $k \in \N$ and any monotone function $f\in \N^\N$, there exists $N\leq \psi(k,f)$ such that
\begin{equation}\label{qv:eq1}
\exists x \in X \left( d(U(x),x) < \dfrac{1}{f(N)+1} \, \land \, \forall n\in [N, f(N)] \, 
\left( \varphi(x,x) < \varphi(x,u_n) + \frac{1}{k+1}\right)\right),
\end{equation}	
where 
\[\fseco{f}(m):=f(\alpha(m,f))=f(16b^2(k+1)^2+8b(k+1))\]
and
\bua
\psi(k,f) &=& \alpha\left(\beta\left(k,\fseco{f}\right), f\right) = 16b^2\left(\beta\left(k,\fseco{f}\right)+1\right)^2+
8b\left(\beta\left(k,\fseco{f}\right)+1\right)\\
&=& 16b^2\left(12b\left(\usf_{\usw_{k,\fseco{f}}}^{(\usr(2k+1))}(0)+1\right)^2+1\right)^2+
8b\left(12b\left(\usf_{\usw_{k,\fseco{f}}}^{(\usr(2k+1))}(0)+1\right)^2+1\right)
\eua
More details:
\bua
\usw_{k,\fseco{f}}(m)&=&\max\left\{\fseco{f}(m),2b(k+1)\right\}\\
 \usf_{\usw_{k,\fseco{f}}}(m) &=& \max\left\{\usw_{k,\fseco{f}}(m)(12b(m+1)^2),\, 12b(m+1)^2 \right\}\\
 &=& \max\left\{\fseco{f}(12b(m+1)^2),2b(k+1), 12b(m+1)^2 \right\}
\eua
}

For the condition $(ii)$ of Proposition \ref{qv:prop2}, we rely on the following result, which is an immediate 
consequence of the more general Proposition~\ref{appbau.prop7} (see the immediate paragraph after the proof of this proposition).

\bprop\label{Wit-ii-section-6}
For all $x\in C$ and all $k,n,p\in \N$, 
$$
\ba{c}\ds \|U(x) - x\| \leq \frac{1}{9b(k+1)^2(p+1)} \,\, 
\wedge \,\, \forall i \in [n,p] \left(\langle x-u_0, U(x) - U(x_i)\rangle \leq \frac{1}{12(k+1)^2}\right)\\[1mm]
\ds \to \,\, \forall m \in [\sigma'(k,n),p] \left(\|u_m -x \| < \frac{1}{k+1}\right),\ea
$$
where $\sigma'(k,n) := \exp\left(\tilde{n}+1 + \lceil \ln(3b^2(k+1)^2)\rceil\right)$, with 
$\tilde{n}:= \max\{n,\,6b^2(k+1)^2 \}$.
\eprop

Applying Proposition~\ref{Wit-ii-section-6}  with $x,n,k$ and $p:=f(\sigma'(k,n))$, we get condition $(ii)$ of Proposition 
\ref{qv:prop2} with the following data: 
\bua
\gamma(k,n,f):=9b(k+1)^2(f(\sigma'(k,n))+1)-1, &
\delta(k):=12(k+1)^2-1,\\
\eta(k,n,f):=f(\sigma'(k,n)) \,\,\text{~and} & 
 M:=\sigma(k,n,f):=\sigma'(k,n).
 \eua
\details{
Let $k,n\in\N$, $f\in \N^\N$ be monotone and $x\in X$ be such that the premise of (ii) of Proposition 
\ref{qv:prop2} holds with 
$\gamma,\delta, \eta,\sigma$ defined as above. Let $M:=\sigma(k,n,f)=\sigma'(k,n)$. We have to prove that 
for all $m\in[M,f(M)]$, $d(u_m,x) < \dfrac{1}{k+1}$. 

We have that 
\bua 
d(U(x),x)\leq \dfrac{1}{\gamma(k,n,f)+1} &\Lra & \|U(x) - x \| \leq \frac{1}{9b(k+1)^2(f(\sigma'(k,n))+1))} \\
\eua
and for all $i\in [n, f(\sigma'(k,n))]$,
\bua
\varphi (x,x) - \varphi (x,u_i)\leq \dfrac{1}{\delta(k)+1} &\Lra & 
\langle x - u_0,U(x)-U(u_i)\rangle \leq \dfrac{1}{\delta(k)+1}\\
&\Lra &\langle x - u_0,U(x)-U(u_i)\rangle \leq \dfrac{1}{24(k+1)^2}.
\eua
Thus, the premise of the implication in Proposition~\ref{Wit-ii-section-6} (applied for $p:=f(\sigma'(k,n))$), 
so the conclusion holds too. It follows that 
for all $m \in [\sigma'(k,n),f(\sigma'(k,n))]=[M,f(M)]$, 
\bua
\|u_m-x \| < \frac{1}{k+1}.
\eua 
Thus, Proposition \ref{qv:prop2}.(ii) holds.
}

As in the case of the mining of Browder's theorem, we can apply Proposition~\ref{qv:prop2} 
to get the following result.

\begin{theorem}[Quantitative Wittmann]\label{app:quant-wittmann}
Under the conditions of Wittmann's theorem, let $b\in\N^*$ be an upper bound on the diameter of $C$. 
Then, for all $k\in \N$ and every monotone function $f:\N 	\to \N$, 
$$\exists N\leq \phi_b(k,f) \,\forall i,j \in [N,f(N)] \, \left( \|u_i-u_j\| < \dfrac{1}{k+1}\right),$$ 
where 
$$\phi_b(k,f):=\sigma'\left(2k+1,\psi\left(48(k+1)^2-1,\fsec{f}\right)\right),$$
with $\sigma'$ and $\psi$ defined above and $\fsec{f}(m)=36b(k+1)^2(f(\sigma'(2k+1,m))+1)-1$.
\end{theorem}
\details{The computations should be verified again, as there were changes in \eqref{Wit-ii-section-6}. Since (i) and (ii) from  Proposition \ref{qv:prop2} are satisfied, we can apply it 
to get that for all $k\in \N$ and every monotone function $f:\N \to \N$,
$$\exists N\leq \phi(k,f) \,\forall i,j \in [N,f(N)] \, \left( \|u_i-u_j\|< \dfrac{1}{k+1}\right),$$ 
where 
\bua \tilde{k} &=& \delta(2k+1) =12(2k+2)^2-1=48(k+1)^2-1\\[1mm]
R &=& \usr(2\tilde{k}+1)  = b^4(2\tilde{k}+2)^2=b^4(3\cdot 2^6\cdot(k+1)^2)^2=
3^2\cdot 2^{12}b^4(k+1)^4\\
&=&  36864 b^4(k+1)^4\\
\fsec{f}(m) &=& \max\{ \gamma(2k+1,m,f),\, \eta(2k+1,m,f)\}\\
&=& \max\{9b(2k+2)^2(f(\sigma'(2k+1,m))+1)-1, f(\sigma'(2k+1,m)) \}\\
&=& 36b(k+1)^2(f(\sigma'(2k+1,m))+1)-1\\[1mm]
\fseco{\fsec{f}}(m) &=& \fsec{f}(16b^2(m+1)^2+8b(m+1))\\[1mm]
\usw_{\tilde{k},\fseco{\fsec{f}}}(m) &=& \max\{\fseco{\fsec{f}}(m),2b(\tilde{k}+1)\}\\
&=& \max\{\fsec{f}(16b^2(m+1)^2+8b(m+1)),2b(\tilde{k}+1)\}\\[1mm]
h(m) &=& \usf_{\usw_{\tilde{k},\fseco{\fsec{f}}}}(m) =\max\{\usw_{\tilde{k},\fseco{\fsec{f}}}(m)(12b(m+1)^2), 12b(m+1)^2\}\\
&=& \max\{\fseco{\fsec{f}}(12b(m+1)^2),2b(\tilde{k}+1), 12b(m+1)^2\}\\
&=& \max\{\fsec{f}(16b^2(12b(m+1)^2+1)^2+8b(12b(m+1)^2+1),  2b(\tilde{k}+1), 12b(m+1)^2\},\\
L_{k,f,b}&:=& \psi\big(\tilde{k},\fsec{f}\big)=
16b^2\big(12b(\usf_{\usw_{\tilde{k},\fseco{\fsec{f}}}}^{(\usr(2\tilde{k}+1))}(0)+1)^2+1\big)^2+
8b\big(12b(\usf_{\usw_{\tilde{k},\fseco{\fsec{f}}}}^{(\usr(2\tilde{k}+1))}(0)+1)^2+1\big)\\
&=& 16b^2\big(12b(h^{(R)}(0)+1)^2+1\big)^2+8b\big(12b(h^{(R)}(0)+1)^2+1\big)\\
\widetilde{L_{k,f,b}} &=& \max \{L_{k,f,b},6b^2(k+1)^2\}
\eua
and 
\bua
\phi_b(k,f) & := & \sigma\left(2k+1,\psi\big(\delta(2k+1),\fsec{f}\big), f\right)=
\sigma'\bigg(2k+1,\psi\big(\delta(2k+1),\fsec{f}\big)\bigg)=\sigma'(2k+1,L_{k,f,b})\\
&=& 3b^2(\widetilde{L_{k,f,b}} +1)(2k+2)^2=12b^2(k+1)^2(\widetilde{L_{k,f,b}} +1),
\eua
Since $h(m)\geq 2b(\tilde{k}+1)$ for all $m$, we get that $L_{k,f,b} > 16b^2\big(2b(\tilde{k}+1)\big)^2>6b^2(k+1)^2$, hence
$\widetilde{L_{k,f,b}} =L_{k,f,b}$.
Thus,
\[
\phi_b(k,f)=12b^2(k+1)^2(L_{k,f,b}+1).\]
}

\section{The analysis of Bauschke's theorem}\label{section_6}

In the sequel we give, for the first time, the mining of Bauschke's generalization of 
Wittmann's theorem to a finite family of nonexpansive mappings \cite{Bauschke(96)}. Throughout 
this section, we fix a natural positive number $\ell$. 

Let $(X,d)$ be a bounded metric space and let  $U_0, \ldots, U_{\ell-1}$ be mappings from $X$ to $X$. 
Consider also mappings $\varphi_0, \ldots, \varphi_{\ell-1}$ from  $X \times X$ to $\R$ and 
$(u_n)_{n\in\N}$ a sequence of elements of  $X$ such that $\lim_n d(U_i(u_n),u_n) = 0$, for 
all $i<\ell$.  We denote by  
$$F: =\bigcap_{i=0}^{\ell-1} Fix(U_i)$$  the set of common fixed points of the mappings $U_0, \ldots, U_{\ell-1}$. 

The following result shows that a general principle similar to the one given by Proposition~\ref{gp.prop_3} 
holds  in this setting too. 

\begin{proposition}\label{bau.prop1}
The theory $\Ms^+$ proves the following mathematical statement. Assume that 
\beq\label{bau.prop1-hyp}
\forall k\in \N \,\exists x \in F\, \forall y\in F\, \forall i<\ell \left(\varphi_i(x,x) < \varphi_i(x,y) + 
\frac{1}{k+1}\right).\eeq
Then 
\beq\label{bau.prop1-con}
\forall k\in \N \,\exists x \in F \,\exists n\in \N \,\forall m\geq n \,\forall i<\ell 
\left( \varphi_i(x,x) < \varphi_i(x,u_m) + \frac{1}{k+1}\right).
\eeq
\end{proposition}
\begin{proof}
Let $k\in \N$ be arbitrary. By \eqref{bau.prop1-hyp}, there exists $\tilde{x}\in F$ such that  $\forall y\in F \forall i<\ell \left(\varphi_i(\tilde{x},\tilde{x}) < \varphi_i(\tilde{x},y) + \frac{1}{k+1}\right)$.\\
By the definition of $F$, we get that
\begin{equation*}
\forall y\in X\, \left(\forall i<\ell \, \forall r\in \N\, \left(d(U_i(y),y)\leq \frac{1}{r+1}\right) \to 
\forall i<\ell \,\left(\varphi_i(\tilde{x},\tilde{x}) < \varphi_i(\tilde{x},y) + \frac{1}{k+1}\right)\right).
\end{equation*}
so, using classical logic,

\begin{equation*}
\forall y\in X\, \exists r\in \N\, \left(\forall i<\ell \,  \left(d(U_i(y),y)\leq \frac{1}{r+1}\right)
\to \forall i<\ell \, \left(\varphi_i(\tilde{x},\tilde{x}) < \varphi_i(\tilde{x},y) + \frac{1}{k+1}\right)\right).
\end{equation*}

Since the formula inside the outer parentheses is equivalent to a $\Sigma$-formula, we can apply $\BC$ to obtain
\begin{equation*}
\exists r\in \N \, \forall y\in X  \left(\forall i<\ell \, \left(d(U_i(y),y)\leq \frac{1}{r+1}\right) \to 
\forall i<\ell \,\left(\varphi_i(\tilde{x},\tilde{x}) < \varphi_i(\tilde{x},y) + \frac{1}{k+1}\right)\right).
\end{equation*}

Take $r_0$ to be one such $r$. Since $\lim_n d(U_i(u_n),u_n) = 0$ for all $i<\ell$, we have that 
\begin{equation*}
\forall i<\ell \,\forall r\in \N \,\exists n\in \N \,\forall m\geq n\, \left(d(U_i(u_m),u_m)\leq \frac{1}{r+1}\right).
\end{equation*}
Note that the bounded quantification ``$\forall i < \ell$'' stands really for a finite conjunction because $\ell$ is a fixed natural number. We really have 
\begin{equation*}
\forall r\in \N \,\bigwedge_{i<\ell} \,\exists n\in \N \,\forall m\geq n\, \left(d(U_i(u_m),u_m)\leq \frac{1}{r+1}\right).
\end{equation*}
Therefore, we easily obtain that
\begin{equation*}
\forall r\in \N \exists n\in \N \,\bigwedge_{i<\ell} \, \forall m\geq n\, \left(d(U_i(u_m),u_m)\leq \frac{1}{r+1}\right).
\end{equation*}
That is,
\begin{equation*}
\forall r\in \N \exists n\in \N \forall m\geq n \forall i<\ell\, \left(d(U_i(u_m),u_m)\leq \frac{1}{r+1}\right).
\end{equation*}
The result now follows by instantiating $r$ by $r_0$. 
\end{proof}

We recapture Proposition~\ref{gp.prop_3} when $\ell = 1$. Thus, Proposition~\ref{bau.prop1} is 
a generalization of Proposition~\ref{gp.prop_3}. Moreover, since the definition of the new set 
$F$ only differs from the original fixed point set by a bounded quantification, it is easy 
to see that we also have a quantitative version of the above proposition. One can argue essentially in the manner of the proof of Proposition~\ref{qv:prop1}.

\begin{proposition}[Quantitative version of \ref{bau.prop1}]\label{bau.prop3}
Suppose that there are monotone functionals $\alpha$ and $\beta$ from $\N \times \N^\N$ to $\N$ satisfying:
\begin{itemize}
\item[\rm{(}a\rm{)}] $\forall k\in \mathbb{N}\, \tilde{\forall} f\in \N^\N\, \exists N\leq \alpha(k,f) \, 
\forall n \in [N,f(N)]\,\forall i<\ell\,\left( d(U_i(u_n),u_n) < \dfrac{1}{k+1}\right)$;
\item[\rm{(}b\rm{)}] $\forall k\in \mathbb{N}\, \tilde{\forall} f \in \N^\N\, \exists N\leq \beta(k,f)\,\exists x\in X
\left(\forall i<\ell \left(d(U_i(x),x)< \dfrac{1}{f(N)+1}\right)\right.$\\
$~\quad\qquad\qquad \wedge\, 
\left.\forall y \in X\, 
\left(\forall i<\ell \left(d(U_i(y),y)\leq \dfrac{1}{N+1}\right) 
\rightarrow 
\forall i<\ell \left(\varphi_i(x,x) < \varphi_i(x,y) + \dfrac{1}{k+1}\right)\right)\right).$
\end{itemize}
Then, for every $k \in \N$ and any monotone function $f\in \N^\N$, there is a natural number 
$N$ with $N\leq \psi(k,f)$ such that
\begin{equation*}
\exists x \in X \left(\forall i<\ell \left(d(U_i(x),x) < \dfrac{1}{f(N)+1}\right) \, \land \, \forall n\in [N, f(N)] \, 
\forall i<\ell 
\left( \varphi_i(x,x) < \varphi_i(x,u_n) + \frac{1}{k+1}\right)\right),
\end{equation*}	
where $\psi(k,f)$ is defined as in Proposition~\ref{qv:prop1}.
\end{proposition}

In analogy with Lemma~\ref{modusponens}, we also have:

\begin{lemma}\label{bau.prop2}
The theory $\Ms$ proves the following mathematical statement. 
Suppose that 
\[
\forall k \in \N \, \exists x \in F \, \exists n \in \N \, \forall m\geq n\, \forall i<\ell 
\left( \varphi_i(x,x) \leq \varphi_i(x,u_m) + \frac{1}{k+1} \right)
\]
and that there is a monotone function $\delta: \N \to \N$ such that, for all $k\in \N$ and $x\in F$, 
\[
\exists n \in \N \, \forall m\geq n \, \forall i<\ell \left( \varphi_i(x,x) \leq \varphi_i(x,u_m) + \frac{1}{\delta(k)+1} \right) \, 
\to \, \exists M \in \N \,\forall m\geq M \left(d(u_m,x) \leq \frac{1}{k+1}\right).
\]
Then, $ (u_n)$ is a Cauchy sequence. 
\end{lemma}

By the same reasoning of the proof of Proposition~\ref{qv:prop2}, we get the  following 
quantitative version of Lemma~\ref{bau.prop2}:

\begin{proposition}[Quantitative version of \ref{bau.prop2}]\label{bau.prop4}
Suppose that there are 
monotone functions $\delta, \psi, \gamma, \eta$ and $\sigma$ satisfying:
\begin{enumerate}
\item[\rm{(}i\rm{)}] $\forall k \in \mathbb{N} \,\tilde{\forall} f \in \N^\N\, \exists N\leq \psi(k,f)$\\
$\exists x\in X \left(\forall i<\ell\left(d(U_i(x),x) < \dfrac{1}{f(N)+1}\right) \,  \land \, \forall n\in [N, f(N)] \, 
\forall i<\ell\left( \varphi_i (x,x) < \varphi_i (x,u_n) + \dfrac{1}{k+1}\right)\right)$\\ 
\text{ and}
\item[\rm{(}ii\rm{)}] $\forall k, n \in \N \, \tilde{\forall} f\in \N^\N \, \forall x \in X $\\ 
$\left[\forall i<\ell \left( d(U_i(x),x)\leq \dfrac{1}{\gamma(k,n,f)+1}\right)\,\land\, \forall m\in [n,\eta(k,n,f)]\, 
\forall i<\ell\left(\varphi_i(x,x) \leq \varphi_i(x,u_m) + \dfrac{1}{\delta(k)+1}\right)\right.$\\
$~\qquad\qquad\qquad\qquad\rightarrow \left.\exists M\leq \sigma(k,n,f)\, \forall m\in[M,f(M)]\, 
\left(d(u_m,x) < \dfrac{1}{k+1}\right)\right]$.\\
\end{enumerate}
Then
\begin{equation}\label{qv:eq4}
\forall k \in \mathbb{N} \,\tilde{\forall} f \in \N^\N\, \exists M\leq \phi(k,f)\, 
\forall m,n\in[M,f(M)]\, \left(d(u_m,u_n) < \dfrac{1}{k+1}\right),
\end{equation}
where $\phi(k,f)$ is defined as in Proposition~\ref{qv:prop2}.
\end{proposition}

Note that the above quantitative versions (Proposition \ref{bau.prop3} and Proposition \ref{bau.prop4}) 
are also true when the metric space is unbounded. This is similar to the situation discussed 
in Subsection~\ref{bfi_explanation}.

\subsection{Mining Bauschke's theorem}

In the following, $X$ is a Hilbert space, $C$ is a nonempty closed convex  bounded subset of 
$X$, $b\in \N^*$ is an upper bound on the diameter of $C$ and  $T_0,\ldots, T_{\ell-1}$ are 
nonexpansive selfmappings of $C$. Let $F$ be the set of common fixed points of the mappings 
$T_0,\ldots, T_{\ell-1}$. 

For each $n\in \N$, define the mapping
\begin{equation}\label{def-Un}
U_n:=T_{n\, \text{mod }\ell}.
\end{equation}
Obviously, $U_i=T_i$ for all $i <\ell$ and $F=\bigcap_{i=0}^{\ell-1} Fix(U_i)=\bigcap_{n\in\N} Fix(U_n)$. 

Let $(\lambda_n)_{n\in\N}$ be a sequence in $(0,1)$ satisfying the conditions 
\[\ba{lll} 
(\rm{C}1) \quad  \lim \lambda_n=0, & (\rm{C}2) \quad \ds\sum_{n=1}^\infty\lambda_n=\infty, & (\rm{C}3[\ell])  
\quad \ds\sum_{n=1}^\infty|\lambda_n-\lambda_{n+\ell}|<\infty.
\ea\]
Given $u\in C$, we define the sequence $(u_n)_{n\in\N}$  by
\beq\label{def-un-Bauschke}
u_0:=u, \quad u_{n+1}:=\lambda_{n+1}u_0 + (1-\lambda_{n+1})U_{n+1}(u_n).
\eeq

The following theorem was proved by Heinz Bauschke in \cite{Bauschke(96)}.

\bthm[Bauschke]\label{Bauschke-thm}
With the above assumptions, suppose furthermore that
\begin{equation}\label{hyp-Fix-Ui}
F=Fix(T_{\ell-1}\cdots T_1T_0)= Fix(T_0 T_{\ell-1}\cdots T_1)=
\cdots =Fix(T_{\ell-2}\cdots T_0 T_{\ell-1}).
\end{equation}
Then $(u_n)$ converges strongly to a common fixed point of 
$T_0,\,\ldots, \, T_{\ell-1}$ {\em (}the closest one to $u${\em )}.
\ethm
Obviously, for $\ell=1$ one gets Wittmann's theorem. 

We remark first that \eqref{hyp-Fix-Ui} is equivalent to 
\begin{equation}\label{hyp-Fix-Ui1}
F=Fix(U_{m+\ell}\cdots U_{m+1})\quad \text{for all~} m\in\N.
\end{equation}

The left-to-right inclusion is obvious. Therefore, \eqref{hyp-Fix-Ui1} holds if, and only if,

\begin{equation*}
\forall m\in\N\left(F\supseteq Fix(U_{m+\ell}\cdots U_{m+1})\right).
\end{equation*}

In order to find the quantitative version of this statement, we display its logical form. 
The above statement can be rewritten as
 
\begin{equation*}
\forall m\in\N \, \forall x \in C \left(\forall r\in\N \, 
\left(\| x - U_{m+\ell}\cdots U_{m+1}(x)\|\leq \frac1{r+1}\right)\rightarrow \forall i <\ell\, 
\forall k\in\N \, \left(\|x - U_i(x)\| < \frac1{k+1}\right)\right).
\end{equation*}
Since the mappings $U_m$ are defined cyclically, the quantification ``$\forall m \in \N$'' 
above can be seen as bounded. Therefore, we get in our formal setting, by using $\BC$, that 
\begin{equation}\label{bau.eq1}
\!\!\!\!\forall k\in\N\, \exists r\in\N \,\forall m\in\N\, \forall x \in C \left( 
\| x - U_{m+\ell}\cdots U_{m+1}(x)\|\leq \frac1{r+1}\rightarrow  \forall i<\ell 
 \left( \| x - U_i(x) \| < \frac1{k+1}\right)\right).
\end{equation}

Hence, for the quantitative version of \eqref{bau.eq1}, we ask for 
a monotone function $\tau:\N\to\N$ satisfying
\beq\label{hyp-Fix-Ui-quant}
\forall k\in\N\, \forall m\in\N\, \forall x \in C \left( 
\|x - U_{m+\ell}\cdots U_{m+1}(x)\|\leq \frac1{\tau(k)+1}\rightarrow  \forall i<\ell 
 \left(\| x - U_i(x)\| < \frac1{k+1}\right)\right).
\eeq 

The quantitative versions of the conditions $(\rm{C}1)$, $(\rm{C}2)$ and $(\rm{C}3[\ell])$ on the sequence $(\lambda_n)$ 
assume the existence of monotone functions $\mu, \nu, \xi:\N\to \N$  satisfying:
\be
\item $\mu$ is a rate of convergence for $(\lambda_n)$ towards zero, that is 
\[(\rm{C}1_q)\quad \forall k \in \N\, \forall n\geq \mu(k) \left(\lambda_n\leq \frac{1}{k+1}\right);\]
\item $\nu$ is a rate of divergence for $\sum_n \lambda_n$, that is 
\[(\rm{C}2_q)\quad \forall k \in \N \, \left(\sum\limits_{j=0}^{\nu(k)} \lambda_j \geq k\right);\]
\item $\xi$ is a Cauchy modulus for the series $\sum_n |\lambda_n-\lambda_{n+\ell}|$, that is 
\[(\rm{C}3[\ell]_q)\quad \forall k \in \N \, \forall n\in \N^*\, \left( \sum\limits_{j=\xi(k)+1}^{\xi(k)+n}|\lambda_j-
	\lambda_{j+\ell}|\leq \frac{1}{k+1} \right).\]
\ee

Note that $\nu(k) \geq k$. In the sequel, we prove some useful properties of the sequence 
$(u_n)$. First, let us remark that, 
for all $n\in \N^*$ and all $m\in\N$,
\begin{equation}\label{appbau.eq1}
\|u_{n+m+\ell}-u_{n+m}\|\leq b\cdot\sum^{n+m}_{j=n} |\lambda_{j+\ell}-\lambda_j| + 
\|u_{n+\ell-1}-u_{n-1}\|\cdot\prod^{n+m}_{j=n}(1-\lambda_{j+\ell}).
\end{equation}
The proof is an easy induction on $m$ (see the proof of \cite[Theorem~3.1]{Bauschke(96)}).

\begin{lemma}\label{lemma-prop-un}
For each $k\in \N$, the following holds:
\be
\item[\rm{(}i\rm{)}]  $\forall n\geq \mu(b(k+1))\left(\|u_{n+1}-U_{n+1}(u_n)\|\leq \frac{1}{k+1}\right)$
\item[\rm{(}ii\rm{)}]  $\forall n\geq \chi(k) \left(\|u_{n+\ell}-u_n\|\leq \frac{1}{k+1}\right)$, where  
$\chi(k):=\nu(\xi(2b(k+1))+1+\ell + \lceil \ln(2b(k+1))\rceil)$.
\item[\rm{(}iii\rm{)}] $\forall n\geq \widetilde{\alpha}(k) 
\left(\|u_n-U_{n+\ell}\cdots U_{n+1}(u_n)\|\leq \frac{1}{k+1}\right)$, 
where $\widetilde{\alpha}(k):=\max\{\mu(2\ell b(k+1)), \chi(2k+1)\}$. 
\item[\rm{(}iv\rm{)}] $\forall n\geq \widehat{\alpha}(k) \,\forall i< \ell\left(\|u_n-U_i(u_n)\| < \frac{1}{k+1}\right)$,
where $\widehat{\alpha}(k):=\widetilde{\alpha}(\tau(k))$, with $\tau$ satisfying \eqref{hyp-Fix-Ui-quant}.
\ee
\end{lemma}
\begin{proof}
\be
\item 
Since $u_{n+1}=\lambda_{n+1}u_0+(1-\lambda_{n+1})U_{n+1}(u_n)$, for $n\geq \mu(b(k+1))$ we have
\begin{equation*}
\|u_{n+1}-U_{n+1}(u_n)\|=\lambda_{n+1}\|u_0-U_{n+1}(u_n)\|\leq \lambda_{n+1}b\leq \frac{1}{k+1}.
\end{equation*}
\item 
Let $N:=\xi(2b(k+1))+1$. Applying \eqref{appbau.eq1} with $n:=N$ and using $(\rm{C}3[\ell]_q)$ and the fact that
$1-x\leq \exp(-x)$ for $x\geq 0$, we get that for all $m\in\N$,
\beq\label{appbau.eq2}
\|u_{N+m+\ell}-u_{N+m}\|\leq \frac{1}{2(k+1)} + b\cdot \exp\left(-\sum_{j=N}^{N+m}\lambda_{j+\ell}\right)
\eeq
\details{
\bua
\|u_{N+m+\ell}-u_{N+m}\|&\leq & b\cdot\sum^{N+m}_{j=N} |\lambda_{j+\ell}-\lambda_j| + 
\|u_{N+\ell-1}-u_{N-1}\|\cdot\prod^{N+m}_{j=N}(1-\lambda_{j+\ell})\\
&\leq & b\frac{1}{2b(k+1)}  + 
b\cdot\prod^{N+m}_{j=N}(1-\lambda_{j+\ell})\\
&&\text{applying~} (\rm{C}3[\ell]_q) \text{~with~} k:=2b(k+1), n:=m+1\\
&\leq & \frac{1}{2(k+1)} + b\cdot \exp(-\sum_{j=N}^{N+m}\lambda_{j+\ell}), 
\eua 
since $1-x\leq \exp(-x)$ for $x\geq 0$.
}
Let $M:=\chi(k)-N=\nu(N+\ell + \lceil \ln(2b(k+1))\rceil)-N$. By $(\rm{C}2_q)$, it follows that 
for all $m\geq M$,
$$\sum_{i=0}^{N+m+\ell} \lambda_i \, \geq \, \sum_{i=0}^{N+M} \lambda_i \, \geq \, N+\ell + 
\lceil \ln(2b(k+1))\rceil 
\, \geq \, \sum_{i=0}^{N+\ell -1} \lambda_i + \ln(2b(k+1)).$$ Therefore, 
$\sum\limits_{i=N}^{N+m} \lambda_{i+\ell} = \sum\limits_{i=N+\ell}^{N+m+\ell} \lambda_i \geq \ln(2b(k+1))$, 
which yields
\begin{equation}\label{appbau.eq3}
b\cdot \exp\left(-\sum_{i=N}^{N+m} \lambda_{i+\ell}\right)\leq \frac{1}{2(k+1)}.
\end{equation}

Now, apply \eqref{appbau.eq2} and \eqref{appbau.eq3} to get (ii).

\item Let $n\geq \widetilde{\alpha}(k)$ be arbitrary. For every $1\leq i\leq \ell$, let 
$S_i:=U_{n+i}\cdots U_{n+1}$. We get
\bua
\|u_n-U_{n+\ell}\cdots U_{n+1}(u_n)\| &=& \|u_n-S_\ell(u_n)\| \leq  \|u_n-u_{n+\ell}\| +  \|u_{n+\ell}-S_\ell(u_n)\|\\
&\leq & \frac{1}{2(k+1)} + \|u_{n+\ell}-S_\ell(u_n)\|
\eua
The inequality is explained by (ii), given that $n\geq \chi(2k+1)$. 

Remark that 
\bua
\|u_{n+\ell}-S_\ell(u_n)\| &\leq & \|u_{n+\ell}-U_{n+\ell}(u_{n+\ell-1})\|+
\|U_{n+\ell}(u_{n+\ell-1})-S_\ell(u_n)\|\\
&\leq &  \|u_{n+\ell}-U_{n+\ell}(u_{n+\ell-1})\|+\|u_{n+\ell-1}-S_{\ell-1}(u_n)\|,
\eua
since $U_{n+\ell}$ is nonexpansive.  Reasoning in the same way, it follows that 
\bua
\|u_{n+\ell}-S_\ell(u_n)\| &\leq & \sum_{i=1}^{\ell}\|u_{n+i}-U_{n+i}(u_{n+i-1})\| \leq 
\frac{\ell}{2\ell(k+1)}=\frac1{2(k+1)},\eua
by (i), given that $n\geq \mu(2\ell b(k+1))$.  
\item Just apply (iii) and \eqref{hyp-Fix-Ui-quant}.
\ee 
\end{proof}

We show now that we can apply our quantitative results (Propositions~\ref{bau.prop3} and \ref{bau.prop4}) by 
considering, for every $i=0,\ldots, \ell-1$,
\beq
\varphi_i(x,y):=\langle x-u_0,U_i(y)\rangle.
\eeq

Note that, as an immediate consequence of Lemma~\ref{lemma-prop-un}.(iv), the functional 
\beq 
\alpha:\N\times\N^\N\to \N, \quad \alpha(k,f):=\widehat{\alpha}(k)
\eeq
satisfies condition $(a)$ of Proposition~\ref{bau.prop3}. 

In the sequel, we show how to compute a functional $\beta$ satisfying condition (b) of Proposition~\ref{bau.prop3}.
We consider the projection onto a different set $F$ than the one in Section~\ref{projection}. Since now $F$ is $\{x\in C\,|\, \forall i< \ell\, (U_i(x)=x)\}$, the only difference to  the analysis of the projection argument is in the innocuous addition of the bounded quantification 
``$\forall i<\ell$.'' We get, using similar arguments to the ones used in the proof of Proposition~\ref{proj3}, 
the following result:

\begin{proposition}\label{proj_adapted}
For any $k\in \N$ and monotone $f:\N \to \N $, there exist $N \in \N$ with 
$N\leq 12b(\usf_{\usw_{k,f}}^{(\usr(k))}(0)+1)^2$  and $x\in C$ such that 
and 
$$\forall i<\ell\left(\|U_i(x)-x \| < \frac{1}{f(N)+1}\right) \wedge \, 
\forall y\in C \left(\forall i<\ell\left(\|U_i(y)-y\|\leq \frac{1}{N+1}\right) \to \langle x-u_0,x-y\rangle < 
\frac{1}{k+1}\right),$$ 
where  $\usr$  is defined by \eqref{def-usr} and and $\usf_{(\cdot)}$ is defined by \eqref{def-usf}.
\end{proposition}

We must change the conclusion of the implication to be compatible with our functions 
$\varphi_i$. I.e., we must replace the conclusion $\langle x-u_0,x-y\rangle < \frac{1}{k+1}$ by $$\forall i<\ell \left(\langle x-u_0, U_i(x)-U_i(y)\rangle < \frac{1}{k+1}\right).$$

This is done, in two steps, in the proposition below. 

\begin{proposition}\label{appbau.prop5}
Let $k\in \N$ and $f:\N \to \N $ be monotone. 
\be 
\item[\rm{(}i\rm{)}] There exist $N_0 \in \N$ with 
$N_0\leq \beta_0(k,f)$  and $x\in C$ such that  
\be
\item[\rm{(}$a_0$\rm{)}] $\|U_i(x)-x \| < \frac{1}{f(N_0)+1}$ for all $i<\ell$, and
\item[\rm{(}$b_0$\rm{)}] for all $y\in C$,
$\forall i<\ell\left(\|U_i(y)-y\|\leq \frac{1}{N_0+1}\right)
\to \forall i< \ell \left(\langle x-u_0,U_i(x)-y\rangle<\frac{1}{k+1}\right);$
\ee
\item[\rm{(}ii\rm{)}] There exist $N \in \N$ with 
$N\leq \beta(k,f)$  and $x\in C$ such that 
\be
\item[\rm{(}a\rm{)}] $\|U_i(x)-x \| < \frac{1}{f(N)+1}$ for all $i<\ell$, and
\item[\rm{(}b\rm{)}] for all $z\in C$, $\forall i<\ell\left(\|U_i(z)-z\|\leq \frac{1}{N+1}\right)
\to \forall i< \ell \left(\langle x-u_0,U_i(x)-U_i(z)\rangle < \frac{1}{k+1}\right),$
\ee
\ee
where 
\beq\label{def-beta-Bauschke}
\beta_0(k,f):=12b(\usf_{\usw_{k,f}}^{(\usr(2k+1))}(0)+1)^2 \quad \mbox{and} \quad \beta(k,f)=3\beta_0(k,g)+2,
\eeq
with $g(m)=f(3m+2)$.
\end{proposition}
\begin{proof}
\be
\item 
Applying Proposition \ref{proj_adapted} to $2k+1$ and to the monotone function 
$\usw_{k,f}$, we get 
 $N_0\leq \beta_0(k,f)$ and $x\in C$ such that 
 $\|U_i(x)-x \| < \frac{1}{\usw_{k,f}(N_0)+1}$ for all
$i<\ell$ and, for all $y\in C$,
\beq\label{appbau.prop5-h}
\forall i<\ell\left(\|U_i(y)-y\|\leq \frac{1}{N_0+1}\right) \to \langle x-u_0,x-y\rangle < 
\frac{1}{2(k+1)}.\eeq

By the definition of $\usw_{k,f}$, we have that, for all $i<\ell$, 
$\|U_i(x)-x\|< \frac{1}{\usw_{k,f}(N_0)+1}\leq \frac{1}{f(N_0)+1}$.
Thus, $(a_0)$ holds. 
Let now $y\in C$ be such that the premise of the implication in $(b_0)$ holds and let $i<\ell$ be arbitrary.
It follows that 
\bua
\langle x-u_0, U_i(x)-y\rangle &=& \langle x-u_0, U_i(x)-x\rangle +\langle x-u_0, x-y\rangle 
\leq  b\cdot \|U_i(x)-x\| + \frac{1}{2(k+1)} \quad \text{by \eqref{appbau.prop5-h}}\\
&\leq & \frac{1}{2(k+1)}+\frac{1}{2(k+1)}=\frac{1}{k+1},
\eua
since $\|U_i(x)-x\|\leq \frac{1}{\usw_{k,f}(N)+1}\leq\frac{1}{2b(k+1)}$. Hence, $(b_0)$ holds too.

\item Apply (i) for $k$ and $g$ to get $N_0\leq \beta_0(k,g)$ and $x\in C$ satisfying $(a_0)$ for $g$ and $(b_0)$. Let $N:=3N_0+2\leq 3\beta_0(k,g)+2=\beta(k,f)$. Then, for all $i<\ell$, we have that 
$\|U_i(x)-x\| < \frac{1}{g(N_0)+1}=\frac{1}{f(N)+1}$, so (a) holds.
In order to prove (b), assume that $z\in C$ is such that 
$\forall i<\ell\left(\|U_i(z)-z\|\leq \frac{1}{N+1}\right)$.

For all $i,j< \ell$, we have that 
\bua
\|U_i(U_j(z))-U_j(z)\|&\leq & \|U_i(U_j(z))-U_i(z)\|+\|U_i(z)-z\|+\|z-U_j(z)\|\\
&\leq & \|U_j(z)-z\|+\frac{2}{N+1}\leq \frac{3}{N+1}=\frac{1}{N_0+1}.
\eua
Thus, we can apply $(b_0)$ for $y:=U_j(z)$, with $j<\ell$ arbitrary, and conclude 
\[\forall i< \ell \,\forall j<\ell \left(\langle x-u_0,U_i(x)-U_j(z)\rangle<\frac{1}{k+1}\right).\]
Take $j:=i$ above to get (b).
\ee
\end{proof}

Thus, we can apply Proposition~\ref{bau.prop3} to get, for every $k \in \N$ and any monotone 
function $f\in \N^\N$, an $N \in \N$ with 
$N\leq \psi(k,f)$ and $x\in C$ such that  
\begin{equation*}\label{appbau.Bauschke-1}
\forall i<\ell \left(\|U_i(x)-x\| < \frac{1}{f(N)+1}\right)  \wedge \, \forall n\in [N, f(N)] 
\, 
\forall i<\ell 
\left( \lan x-u_0, U_i(x)-U_i(u_n)\ran  <  \frac{1}{k+1}\right),
\end{equation*}	
where 
\beq\label{def-psi-Bauschke}
\psi(k,f):= \alpha\left(\beta\left(k,\fseco{f}\right),f\right) = \widehat{\alpha}\left(\beta\left(k,\fseco{f}\right)\right), \quad \text{with~} \fseco{f}(m):=f(\alpha(m,f)).
\eeq 
Hence, condition (a) of Proposition~\ref{bau.prop4} is satisfied with $\psi$ as above.

Next, we present the quantitative result of the main combinatorial step in Bauschke's proof, 
slightly adapted to fit into the general principle.

\begin{proposition}\label{appbau.prop7}
Assume that $k, n, p\in \N$ and $x\in C$ satisfy
\[\forall i< \ell \left(\|U_i(x)-x\|\leq \frac{1}{9b(k+1)^2(p+1)}\right) \wedge\, 
	\forall r\in [n,p]\,\forall i<\ell\left(\langle x-u_0,U_i(x)-U_i(u_r)\rangle \leq \frac{1}{12(k+1)^2}\right).\]
Then 
\[\forall m\in[\sigma'(k,n),p]\left(\|u_m-x\| < \frac{1}{k+1}\right),\]
where $\sigma'(k,n):=\nu(\widetilde{n}+1 + \lceil \ln(3b^2(k+1)^2)\rceil)$ with 
$\widetilde{n}:=\max\{n,\,\mu(6b^2(k+1)^2) \}$.
\end{proposition}
\begin{proof}
First, let us remark that for all $r\in\N$ and $x\in C$,
\bua
\|u_{r+1}-x\|^2 &=& \|\lambda_{r+1}u_0+(1-\lambda_{r+1})U_{r+1}(u_r)-x\|^2=
\|\lambda_{r+1}(u_0-x)+(1-\lambda_{r+1})(U_{r+1}(u_r)-x)\|^2\\
&=& \lambda_{r+1}^2\|u_0-x\|^2+2\lambda_{r+1}(1-\lambda_{r+1})\lan u_0-x, U_{r+1}(u_r)-x\ran+
(1-\lambda_{r+1})^2 \|U_{r+1}(u_r)-x\|^2\\
& =& \lambda_{r+1}^2\|u_0-x\|^2+2\lambda_{r+1}(1-\lambda_{r+1})(\lan x-u_0, x-U_{r+1}(x)\ran+
\lan x-u_0, U_{r+1}(x)-U_{r+1}(u_r)\ran\\
&& +(1-\lambda_{r+1})^2 \|(U_{r+1}(u_r)-U_{r+1}(x))+(U_{r+1}(x)-x)\|^2\\
& \leq & \lambda_{r+1}^2b^2+2b\lambda_{r+1}(1-\lambda_{r+1})\|x-U_{r+1}(x)\|+2\lambda_{r+1}(1-\lambda_{r+1})
\lan x-u_0, U_{r+1}(x)-U_{r+1}(u_r)\ran\\
&& + (1-\lambda_{r+1})^2(\|u_r-x\|^2+2\|u_r-x\|\|U_{r+1}(x)-x\|+\|U_{r+1}(x)-x\|^2)\\
& \leq & \lambda_{r+1}^2b^2+2b\lambda_{r+1}(1-\lambda_{r+1})\|x-U_{r+1}(x)\|+2\lambda_{r+1}(1-\lambda_{r+1})
\lan x-u_0, U_{r+1}(x)-U_{r+1}(u_r)\ran\\
&& + 3b(1-\lambda_{r+1})^2\|x-U_{r+1}(x)\| + (1-\lambda_{r+1})^2\|u_r-x\|^2\\
& \leq & \lambda_{r+1}^2b^2+2\lambda_{r+1}(1-\lambda_{r+1})
\lan x-u_0, U_{r+1}(x)-U_{r+1}(u_r)\ran \\
&& +(2b\lambda_{r+1}(1-\lambda_{r+1})+3b(1-\lambda_{r+1})^2)\|x-U_{r+1}(x)\| + 
(1-\lambda_{r+1})^2\|u_r-x\|^2\\
&\leq & \lambda_{r+1}(\lambda_{r+1}b^2+ 2\lan x-u_0, U_{r+1}(x)-U_{r+1}(u_r)\ran)\\
&& + 3b(1-\lambda_{r+1})\|x-U_{r+1}(x)\|+(1-\lambda_{r+1})\|u_r-x\|^2
\eua
\details{$2b\lambda_{r+1}(1-\lambda_{r+1})+3b(1-\lambda_{r+1})^2=b(1-\lambda_{r+1})
(2\lambda_{r+1}+3(1-\lambda_{r+1}))=b(1-\lambda_{r+1})(3-\lambda_{r+1})$}

Fix $k, n, p\in \N$ and $x\in C$, and assume that they satisfy the hypothesis of the theorem. Take $r \in \N$ with $r\in [\tilde{n},p]\se [n,p]$. Then 
$\lambda_{r+1}\leq \frac{1}{6b^2(k+1)^2}$, since $r+1>\tilde{n}\geq \mu(6b^2(k+1)^2)$ and $\mu$ 
satisfies $(\rm{C}1)_q$. Moreover, $\langle x-u_0, U_{r+1}(x)-U_{r+1}(u_r)\rangle 
\leq \frac{1}{12(k+1)^2}$ by hypothesis. Hence, 
\[\lambda_{r+1}b^2+ 2\lan x-u_0, U_{r+1}(x)-U_{r+1}(u_r)\ran \leq \frac{1}{3(k+1)^2}.\]
Furthermore, as $\|x-U_{r+1}(x)\|\leq \frac{1}{9b(k+1)^2(p+1)}$, we get that
\bua
\|u_{r+1}-x\|^2 \leq \lambda_{r+1}\frac1{3(k+1)^2}+(1-\lambda_{r+1})\frac1{3(k+1)^2(p+1)}+
(1-\lambda_{r+1})\|u_r-x\|^2.
\eua

By induction on $m$, we can prove that for $m\in [\tilde{n}+1, p]$,
\begin{equation}\label{appbau.eq13}
\|u_m-x\|^2\leq \frac{1}{3(k+1)^2}+A_m\frac{1}{3(k+1)^2(p+1)}
+ B_m\|u_{\tilde{n}}-x\|^2
\end{equation}
where   $A_m=\sum_{i=\tilde{n}+1}^{m}\prod_{j=i}^{m}(1-\lambda_j)$ and  $B_m=\prod_{j=\tilde{n}+1}^{m}(1-\lambda_j)$.

\details{We prove \eqref{appbau.eq13} by induction on $m$.
$m:=\tilde{n}+1$: We have that 
\bua 
\|u_{\tilde{n}+1}-x\|^2 & \leq &  \lambda_{\tilde{n}+1}\frac1{3(k+1)^2}+(1-\lambda_{\tilde{n}+1})\frac1{3(k+1)^2(p+1)}+
(1-\lambda_{\tilde{n}+1})\|u_{\tilde{n}}-x\|^2.
\eua
Furthermore,
\bua
A_{\tilde{n}+1} & = & \lambda_{\tilde{n}+1}+\sum_{i=\tilde{n}+2}^{\tilde{n}+1}\lambda_{i-1}\prod_{j=i}^{\tilde{n}+1}
(1-\lambda_j)=\lambda_{\tilde{n}+1},\\
B_{\tilde{n}+1}  & = & \sum_{i=\tilde{n}+1}^{{\tilde{n}+1}}\prod_{j=i}^{{\tilde{n}+1}}(1-\lambda_j)=1-\lambda_{\tilde{n}+1}\\
C_{\tilde{n}+1}& = &\prod_{j=\tilde{n}+1}^{{\tilde{n}+1}}(1-\lambda_j)=1-\lambda_{\tilde{n}+1}.
\eua
$m:=\tilde{n}+2$: We have that 
\bua
\|u_{\tilde{n}+2}-x\|^2 & \leq &  \lambda_{\tilde{n}+2}\frac1{3(k+1)^2}+(1-\lambda_{\tilde{n}+2})\frac1{3(k+1)^2(p+1)}+
(1-\lambda_{\tilde{n}+2})\|u_{\tilde{n}+1}-x\|^2\\
& \leq & \lambda_{\tilde{n}+2}\frac1{3(k+1)^2}+(1-\lambda_{\tilde{n}+2})\frac1{3(k+1)^2(p+1)}\\
&& + (1-\lambda_{\tilde{n}+2})\left(\lambda_{\tilde{n}+1}\frac1{3(k+1)^2}+(1-\lambda_{\tilde{n}+1})\frac1{3(k+1)^2(p+1)}+
(1-\lambda_{\tilde{n}+1})\|u_{\tilde{n}}-x\|^2\right)\\
&=& \left(\lambda_{\tilde{n}+2}+(1-\lambda_{\tilde{n}+2})\lambda_{\tilde{n}+1}\right)\frac1{3(k+1)^2}+
\left((1-\lambda_{\tilde{n}+2})+(1-\lambda_{\tilde{n}+2})(1-\lambda_{\tilde{n}+1})\right)\frac1{3(k+1)^2(p+1)}\\
&& +(1-\lambda_{\tilde{n}+2})(1-\lambda_{\tilde{n}+1})\|u_{\tilde{n}}-x\|^2
\eua
Furthermore, 
\bua
A_{\tilde{n}+2} & = & \lambda_{\tilde{n}+2}+\sum_{i=\tilde{n}+2}^{\tilde{n}+2}\lambda_{i-1}\prod_{j=i}^{\tilde{n}+2}
(1-\lambda_j)=\lambda_{\tilde{n}+2}+\lambda_{\tilde{n}+1}(1-\lambda_{\tilde{n}+2}),\\
B_{\tilde{n}+2}  & = & \sum_{i=\tilde{n}+1}^{\tilde{n}+2}\prod_{j=i}^{{\tilde{n}+2}}(1-\lambda_j)=1-\lambda_{\tilde{n}+1}=
 \prod_{j=\tilde{n}+1}^{{\tilde{n}+2}}(1-\lambda_j)+\prod_{j=\tilde{n}+2}^{{\tilde{n}+2}}(1-\lambda_j)\\
&=& (1-\lambda_{\tilde{n}+1})(1-\lambda_{\tilde{n}+2})+(1-\lambda_{\tilde{n}+2})\\
C_{\tilde{n}+2}& = &\prod_{j=\tilde{n}+1}^{{\tilde{n}+2}}(1-\lambda_j)=(1-\lambda_{\tilde{n}+1})(1-\lambda_{\tilde{n}+2}).
\eua
$m\Ra m+1$:
We have that 
\bua
|u_{m+1}-x\|^2 & \leq &  \lambda_{m+1}\frac1{3(k+1)^2}+(1-\lambda_{m+1})\frac1{3(k+1)^2(p+1)}+
(1-\lambda_{m+1})\|u_m-x\|^2 \\
&\leq &  \lambda_{m+1}\frac1{3(k+1)^2}+(1-\lambda_{m+1})\frac1{3(k+1)^2(p+1)}\\
&&+
(1-\lambda_{m+1})\left(A_m\frac{1}{3(k+1)^2}+B_m\frac{1}{3(k+1)^2(p+1)} 
+ C_m\|u_{\tilde{n}}-u\|^2\right)\\
&=& \left(\lambda_{m+1}+(1-\lambda_{m+1})A_m\right)\frac{1}{3(k+1)^2}+
\left((1-\lambda_{m+1})+(1-\lambda_{m+1})B_m\right)\frac{1}{3(k+1)^2(p+1)}\\
&& +(1-\lambda_{m+1})C_m\|u_{\tilde{n}}-u\|^2
\eua
Remark that
\bua
\lambda_{m+1}+(1-\lambda_{m+1})A_m &=& \lambda_{m+1}+
(1-\lambda_{m+1})\left(\lambda_m+\sum_{i=\tilde{n}+2}^m\lambda_{i-1}\prod_{j=i}^{m}(1-\lambda_j)\right)\\
&=& \lambda_{m+1}+\lambda_{m}\prod_{j=m+1}^{m+1}(1-\lambda_j)+\sum_{i=\tilde{n}+2}^m\lambda_{i-1}\prod_{j=i}^{m+1}(1-\lambda_j)\\
&=& A_{m+1}\\
(1-\lambda_{m+1})+(1-\lambda_{m+1})B_m &=& (1-\lambda_{m+1})+(1-\lambda_{m+1})\sum_{i=\tilde{n}+1}^{m}
\prod_{j=i}^{m}(1-\lambda_j)\\
&=& \prod_{j=m+1}^{m+1}(1-\lambda_j)+\sum_{i=\tilde{n}+1}^{m}
\prod_{j=i}^{m+1}(1-\lambda_j)\\
&=& B_{m+1}\\
(1-\lambda_{m+1})C_m &=& (1-\lambda_{m+1})\prod_{j=\tilde{n}+1}^{m}(1-\lambda_j)=C_{m+1}.
\eua
}

\details{I will prove, by induction on $m$, that for all $m\in [\tilde{n}+1,p]$ the following inequality holds:
\begin{equation}\label{appbau.neweq13}
\|u_m-x\|^2\leq \frac{1}{3(k+1)^2}+B_m\frac{1}{3(k+1)^2(p+1)}+ C_m\|u_{\tilde{n}}-x\|^2,
\end{equation}
where $B_m=\sum_{i=\tilde{n}+1}^{m}\prod_{j=i}^{m}(1-\lambda_j)$ and $C_m=\prod_{j=\tilde{n}+1}^{m}(1-\lambda_j)$.\\

We showed that for any $r\in [\tilde{n}, p]$
\begin{equation}\label{ineq}
\|u_{r+1}-x\|^2 \leq \lambda_{r+1}\frac{1}{3(k+1)^2}+(1-\lambda_{r+1})\frac{1}{3(k+1)^2(p+1)}+(1-\lambda_{r+1})\|u_r-x\|^2.
\end{equation}

$m=\tilde{n}+1$: By \eqref{ineq} applied to $r=\tilde{n}$,
\begin{align*}
\|u_{\tilde{n}+1}-x\|^2 & \leq \lambda_{\tilde{n}+1}\frac1{3(k+1)^2}+(1-\lambda_{\tilde{n}+1})\frac1{3(k+1)^2(p+1)}+ (1-\lambda_{\tilde{n}+1})\|u_{\tilde{n}}-x\|^2\leq \\
&\leq \frac1{3(k+1)^2}+(1-\lambda_{\tilde{n}+1})\frac1{3(k+1)^2(p+1)}+
(1-\lambda_{\tilde{n}+1})\|u_{\tilde{n}}-x\|^2,
\end{align*}
as we wanted.\\

Consider $m\in[\tilde{n}+1,p[$ satisfying \eqref{appbau.neweq13}. We want to see that \eqref{appbau.neweq13} holds for $m+1$. Starting with \eqref{ineq} applied to $r=m$ we have the following:
\begin{align*}
\|u_{m+1}-x\|^2 &\leq \lambda_{m+1}\frac1{3(k+1)^2}+(1-\lambda_{m+1})\frac1{3(k+1)^2(p+1)}+ (1-\lambda_{m+1})\|u_m-x\|^2\leq_{\text{I.H.}}\\
&\leq_{\text{I.H.}} \lambda_{m+1}\frac1{3(k+1)^2}+(1-\lambda_{m+1})\frac1{3(k+1)^2(p+1)} \\
&\qquad + (1-\lambda_{m+1})\big[\frac{1}{3(k+1)^2}+B_m\frac{1}{3(k+1)^2(p+1)}
+ C_m\|u_{\tilde{n}}-x\|^2\big]=\\
&=\frac1{3(k+1)^2}\big[\lambda_{m+1} + (1-\lambda_{m+1})\big] + \frac{1}{3(k+1)^2(p+1)}\big[(1-\lambda_{m+1}) + (1-\lambda_{m+1})B_m\big]\\
&\qquad +(1-\lambda_{m+1})C_m\|u_{\tilde{n}}-x\|^2=\\
&=\frac{1}{3(k+1)^2}+B_{m+1}\frac{1}{3(k+1)^2(p+1)}+ C_{m+1}\|u_{\tilde{n}}-x\|^2
\end{align*}
For the last equality we used
\begin{align*}
&(1-\lambda_{m+1}) + (1-\lambda_{m+1})B_m=B_{m+1}\\
&(1-\lambda_{m+1})C_m=C_{m+1}
\end{align*}
}

Let $m\in[\sigma'(k,n),p]$ be arbitrary. Since $m\leq p$, we have that $A_m\leq m-\tilde{n}<p+1$. 
As $\sigma'(k,n)\geq \tilde{n}+1$, it follows that $m\in[\tilde{n}+1, p]$, so 
we can apply \eqref{appbau.eq13} and get 
\begin{equation}\label{appbau.eq131}
\|u_m-x\|^2 < \frac{1}{3(k+1)^2}+\frac{1}{3(k+1)^2}+ B_m\|u_{\tilde{n}}-x\|^2
\end{equation}

Now, because $m\geq \sigma'(k,n)$, we get 
\begin{equation*}
\sum_{j=0}^{m}\lambda_j \geq \sum_{j=0}^{\sigma'(k,n)}\lambda_j\geq \tilde{n}+1 +
\ln(3b^2(k+1)^2)\geq \sum_{j=0}^{\tilde{n}}\lambda_j+\ln(3b^2(k+1)^2).
\end{equation*}
Therefore, $\sum\limits_{j=\tilde{n}+1}^{m}\lambda_j\geq \ln(3b^2(k+1)^2)$. 
This, in turn, implies
\begin{equation}\label{appbau.eq14}
B_m\|u_{\tilde{n}}-x\|^2\leq b^2\exp\left(-\sum_{j=\tilde{n}+1}^{m}\lambda_j\right)
\leq \frac{1}{3(k+1)^2}.
\end{equation}
The conclusion follows.
\end{proof}

Note that Proposition~\ref{Wit-ii-section-6} is the particular case of the above proposition. One can see this by putting $\ell=1$ and $\lambda_n=\frac{1}{n+1}$, and taking into account that $\mu(n)=n$ is a rate of convergence towards $0$ for the sequence $\left(\frac{1}{n+1}\right)$ and that $\nu(n)=\exp(n)$ is a rate of divergence for $\sum_n\frac1{n+1}$.

\medbreak

We are now in position to apply Proposition~\ref{appbau.prop7} with $x$, $n$, $k$ and $p:= f(\sigma'(k,n))$ in order to obtain 
condition (ii) of Proposition~\ref{bau.prop4}. Just let
\bua
\gamma(k,n,f) := 9b(k+1)^2(f(\sigma'(k,n))+1)-1, & \delta(k) := 12(k+1)^2-1, \\
\eta(k,n,f) := f(\sigma'(k,n)) &  \,\,\text{~and} &  M := \sigma(k,n,f):= \sigma'(k,n).
\eua

\details{
Let $k,n\in\N$, $f\in \N^\N$ be monotone and $x\in X$ be such that the premise of (ii) of 
Proposition~\ref{bau.prop4} holds with 
$\gamma,\delta, \eta,\sigma$ defined as above. Let $M:=\sigma(k,n,f)=\sigma'(k,n)$. We have to prove that 
for all $m\in[M,f(M)]$, $\|u_m-x\| < \frac{1}{k+1}$. 

We show that we can apply Proposition~\ref{appbau.prop7}. We verify first that the two hypotheses hold.

We have that $\forall i<\ell \left(\|U_i(x)-x\|\leq \frac{1}{\gamma(k,n,f)+1}\right)$ by the 
first premise of (ii) of Proposition~\ref{bau.prop4}. This is equivalent with 
$\forall i<\ell \left(d(U_i(x),x)\leq \frac{1}{9b(k+1)^2(p+1)}\right)$.

Let $m\in [n,p]=[n,f(\sigma'(k, n))]=[n,\eta(k, n, f)]$ be arbitrary. By the second premise of 
(ii) of Proposition~\ref{bau.prop4}, we have that 
$$\forall i<\ell\left(\varphi_i(x,x) \leq \varphi_i(x,u_m) + \dfrac{1}{\delta(k)+1}\right),$$
that is 
$$\forall i<\ell\left(\langle x-u_0,U_i(x)-U_i(u_m)\rangle < \frac{1}{12(k+1)^2}\right).$$

Hence, the two hypotheses of Proposition~\ref{appbau.prop7} hold, so we can apply it to get that 
\[\forall m\in[\sigma'(k,n),p]\left(\|u_m-x\|\leq \frac{1}{k+1}\right),\]
that is
\[\forall m\in[\sigma'(k,n),f(\sigma'(k, n))]\left(\|u_m-x\|\leq \frac{1}{k+1}\right),\]
that is, by the definition of $M$,
\[\forall m\in[M,f(M)]\left(\|u_m-x\|\leq \frac{1}{k+1}\right).\]
Furthermore, $M=\sigma(k,n,f)$.
}

Finally, we apply Proposition~\ref{bau.prop4} to obtain the metastable version of Bauschke's theorem.

\begin{theorem}\label{appbau.teo1}
Let $X$ be a Hilbert space, $C$ be a nonempty closed convex  bounded subset of 
$X$, $b\in \N^*$ be an upper bound on the diameter of $C$ and  $T_0,\ldots, T_{\ell-1}$ be  
nonexpansive selfmappings of $C$.

For each $n\in \N$, let $U_n$ be the mapping defined by \eqref{def-Un} and assume that $\tau: \N\to\N$ is a monotone function $\tau: \N\to\N$ satisfying \eqref{hyp-Fix-Ui-quant}. 
Consider a sequence $(\lambda_n)$ in $(0,1)$ and monotone functions $\mu,\nu,\xi:\N\to\N$ such that
$(\rm{C}1_q), (\rm{C}2_q)$ and $(\rm{C}3[\ell]_q)$ hold.
Let $u_0\in C$ be given and $(u_n)$ be the iteration defined by \eqref{def-un-Bauschke}.
 
Then, for all $k\in \N$ and every monotone function $f:\N\to \N$,
\begin{equation}
\exists N\leq \phi_b(k,f) \,\forall i,j \in [N,f(N)] \left( \|u_i-u_j\|\leq \dfrac{1}{k+1}\right),
\end{equation}
where $$\phi_b(k,f):=\sigma'\left(2k+1,\psi\left(48(k+1)^2-1,\fsec{f}\right)\right),$$
with $\sigma'$ defined in Proposition~\ref{appbau.prop7}, $\psi$ defined by \eqref{def-psi-Bauschke} and
$\fsec{f}(m)=36b(k+1)^2(f(\sigma'(2k+1,m))+1)-1$.
\end{theorem}
\details{
\bua
\phi_b(k,f) & := & \sigma\left(2k+1,\psi\big(\delta(2k+1),\fsec{f}\big), f\right)=
\sigma\left(2k+1,\psi\big(\tilde{k},\fsec{f}\big), f\right)\\
&=& \sigma'\bigg(2k+1,\psi\big(\delta{2k+1},\fsec{f}\big)\bigg)=\sigma'(2k+1,L_{k,f,b})\\
\eua
\bua 
\delta(2k+1) &= & 12(2k+2)^2-1=48(k+1)^2-1\\ [1mm]
\fsec{f}(m) &=& \max\{ \gamma(2k+1,m,f),\, \eta(2k+1,m,f)\}\\
&=& \max\{9b(2k+2)^2(f(\sigma'(2k+1,m))+1)-1, f(\sigma'(2k+1,m)) \}\\
&=& 36b(k+1)^2(f(\sigma'(2k+1,m))+1)-1\\
\eua
}

It is well-known that condition $(\rm{C}2)$, $\sum_n \lambda_n=\infty$, is equivalent with 
\bua 
(\rm{C}4) & \ds \prod_{n\to\infty}(1-\lambda_n)=0.
\eua
Hence, one can obtain general quantitative results by using, instead of a rate of divergence $\nu$ for
$\sum_n\lambda_n$, the quantitative version of $(\rm{C}4)$, asserting the existence of a rate of convergence $\theta$ for 
$\prod_{n\to\infty}(1-\lambda_n)$:
\bua
(\rm{C}4_q) &  \ds \forall k \in \N\, \left(\prod_{i=1}^{\theta(k)}(1-\lambda_i)\leq \frac{1}{k+1}\right).
\eua
This was done in \cite{KohlenbachLeustean(12)}, where Kohlenbach and the second author obtained rates of metastability 
for the generalization of Wittmann's theorem to 
CAT(0) spaces using both $(\rm{C}2_q)$ and $(\rm{C}4_q)$. As Kohlenbach remarked in \cite{Kohlenbach(11)}, for $\lambda_n=\frac1{n+1}$, 
one has an exponential $\nu$ and a linear $\theta$, so one gets, by using $(\rm{C}4_q)$, a quadratic rate of asymptotic regularity for the Halpern 
iteration (see \cite[Lemma 3.1]{Kohlenbach(11)}), significantly improving the exponential bound 
obtained in \cite{Leustean(07)}, where $(\rm{C}2_q)$ is used. As a consequence, better rates of 
metastability for Wittmann's theorem are obtained in \cite{Kohlenbach(11),KohlenbachLeustean(12)} 
compared to our Theorem~\ref{app:quant-wittmann}.

One can replace $(\rm{C}2_q)$ with $(\rm{C}4_q)$ also in the quantitative analysis of Bauschke's theorem and prove 
corresponding versions of Proposition~\ref{appbau.prop7} and Theorem~\ref{appbau.teo1} having as a consequence, for $\ell=1$, a 
metastable version of Wittmann's theorem with bounds similar to the ones computed in \cite{Kohlenbach(11),KohlenbachLeustean(12)}.
The drawback is that the proof of this new version becomes much more technical. In this paper,  
the focus is not on the complexity of the bounds, but on the method used to obtain them, so we think 
that it is better to keep the computations as simple as possible.

\mbox{}

We finish by pointing out that recently, K\"ornlein \cite{Koernlein(16)} applied proof mining methods to 
obtain quantitative versions of strong convergence results, due to Yamada \cite{Yamada(01)}, for the hybrid steepest descent method. 
As a consequence, he also obtains a metastable version of Bauschke's theorem.
A direction for future research would be to explore if the methods developed in this paper, based 
on bounded functional interpretation, can be used to obtain similar results with those in \cite{Koernlein(16)}.

\section{Coda}

The remote origins of this paper can be traced to an early intention of the first author to eventually use the bounded functional interpretation in proof mining. When the third author approached the first author for a possible PhD supervision, an opportunity arose to carry out this plan. We decided to  apply the bounded functional interpretation first to a result already mined. Of the myriad of such results, we have to thank Ulrich Kohlenbach for suggesting Browder's strong convergence theorem. The third author found that the mining of the projection argument in Browder's proof turns out to be simpler and more elegant with the bounded functional interpretation. The first author was, however, still disatisfied with the theoretical standing of the sequential weak compactness 
argument in Browder's proof. In the summer of 2017, he came up with the idea of using Heine-Borel compactness instead. The second author visited Lisbon for a 
week in October 2017. Together with the third author, they generalized the Heine-Borel compactness argument to metric spaces and proposed themselves to apply this general principle to the minings of the theorems of Wittmann and Bauschke. From February to April 2018, the third author visited the second author in Bucharest, where more work was done. The end result of these endeavours is the present article.

\section*{Acknowledgements}

Fernando Ferreira and Pedro Pinto were partially supported by Portuguese funding from FCT, Funda\c c\~ao para a Ci\^encia e a Tecnologia, under the project CMAFcIO: UID/MAT/04561/2019. Lauren\c{t}iu Leu\c{s}tean was partially supported by a grant of the Romanian Ministry of Research and Innovation, Program 1 - Development of the National RDI System, Subprogram 1.2 - Institutional Performance -  Projects for Funding the Excellence in RDI, contract number 15PFE/2018. Pedro Pinto also benefited from a FCT doctoral grant PD/BD/52645/2014, under the program Lisbon Mathematics PhD.

\bibliographystyle{plain}

\end{document}